%% file: va_cg_fact.tex
\documentclass[a4paper]{article}
\pdfoutput=1
\input{texenvironment.tex}
\input{commands.tex}

\usepackage{tikz}
\usetikzlibrary{positioning}
\usepackage{tikz-cd}

\title{Vertex Algebras and Costello-Gwilliam Factorization Algebras}
\author{Daniel Bruegmann}
\date{\today}
\begin{document}
\maketitle
\begin{abstract}
  Vertex algebras and factorization algebras are two approaches to chiral conformal field theory. 
  Costello and Gwilliam describe how every holomorphic factorization algebra on the plane of complex numbers satisfying certain assumptions gives rise to a Z-graded vertex algebra. 
  They construct some models of chiral conformal theory as factorization algebras. 
  We attach a factorization algebra to every Z-graded vertex algebra.
\end{abstract}
\input{intro}
\tableofcontents
\input{mero_prefa}
\input{fact}
\bibliographystyle{plain}
\bibliography{va_cg_fact}  
\end{document}

%% file: texenvironment.tex
\usepackage{amsmath}
\usepackage{amsthm}
\usepackage{amssymb}
\usepackage{mathtools}
\usepackage{hyperref}

%% file: commands.tex
\newcommand{\Section}[1]{\section{#1}}
\newcommand{\Subsection}[1]{\subsection{#1}}
\newcommand{\RelatedWork}[1]{\paragraph{Related work.}}
\newcommand{\ootpi}{\frac{1}{2\pi i}}
\newcommand{\V}{\mathcal{V}}
\newcommand{\Vbar}{\overline{\V}}
\newcommand{\Vbarbb}{\overline{\V}_{bb}}
\newcommand{\R}{\mathbf{R}}
\newcommand{\C}{\mathbf{C}}

\newcommand{\Z}{\mathbf{Z}}
\newcommand{\D}{\mathbf{D}}

\newcommand{\Hom}{\mathrm{Hom}}
\newcommand{\End}{\mathrm{End}}
\newcommand{\op}{\mathrm{op}}
\newcommand{\blank}{\text{\textendash}}
\newcommand{\after}{\circ}
\newcommand{\barotimes}{\mathbin{\bar{\otimes}}}
\newcommand{\labelledMapMapsto}[5]{%
\begin{align*}%
  #1 : #2 &\longrightarrow #3\\%
       #4 &\longmapsto #5%
\end{align*}%
}
\newcommand{\displayperiod}{\,\text{.}}
\newcommand{\displaycomma}{\,\text{,}}
\newcommand{\MapMapsto}[4]{%
\begin{align*}%
  #1 &\longrightarrow #2\\%
  #3 &\longmapsto #4%
\end{align*}%
}
\DeclareMathOperator{\im}{im}
\DeclareMathOperator{\pr}{pr}
\DeclareMathOperator{\ev}{ev}
\DeclareMathOperator{\res}{res}
\DeclareMathOperator{\id}{id} 
\DeclareMathOperator{\const}{const}
\DeclareMathOperator{\BVS}{BVS}
\DeclareMathOperator{\CBVS}{CBVS}
\DeclareMathOperator{\Hilb}{Hilb}
\newcommand{\vac}{|0\rangle}
\DeclareMathOperator*{\colim}{colim}

\theoremstyle{definition}
\newtheorem{definition}{Definition}[subsection]
\newtheorem{remark}[definition]{Remark}
\newtheorem{example}[definition]{Example}
\newtheorem*{example*}{Example}

\theoremstyle{plain}
\newtheorem{corollary}[definition]{Corollary}
\newtheorem{proposition}[definition]{Proposition}
\newtheorem{lemma}[definition]{Lemma}
\newtheorem{theorem}[definition]{Theorem}
\newtheorem*{theorem*}{Theorem}
\newenvironment{ctikzcd}{%
\begin{center}%
\begin{tikzcd}\ignorespaces}{%
\end{tikzcd}%
\end{center}\ignorespacesafterend}

\newcommand{\link}[1]{\href{#1}{\tt \nolinkurl{#1}}}
\newcommand{\arxiv}[1]{\href{https://arxiv.org/abs/#1}{{\tt \nolinkurl{#1}}}}

%% file: intro.tex
\noindent Vertex algebras and factorization algebras are two approaches to chiral conformal field theory. 
The subject of vertex algebras is well-developed.
Some aspects of it have been reformulated in terms of the geometry of~$\C$, the plane of complex numbers, or other Riemann surfaces. 
Factorization algebras as developed by Costello and Gwilliam~\cite{CostelloGwilliam} are a more general approach to quantum field theory which applies to all kinds of geometries, including higher dimensional manifolds. 
This article compares vertex algebras and factorization algebras on~$\C$. 

Many models of chiral conformal field theory have been constructed as vertex algebras. 
Some of these have been constructed as factorization algebras, too. 
Costello and Gwilliam describe a procedure to obtain vertex algebras from suitable factorization algebras on~$\C$. 
I provide a one-sided inverse to this procedure.
\begin{theorem*}  
  If~$\V$ is a vertex algebra, then there is a factorization algebra~$\mathbf{F}\V$ on~$\C$ whose associated vertex algebra is isomorphic to~$\V$.
\end{theorem*}
See Theorem~\ref{theorem:MoreDetailedMainTheorem} for a more detailed statement of our main result.
In particular, every vertex algebra arises from a factorization algebra. 
This was known for the universal affine vertex algebras~\cite{CostelloGwilliam} and the Virasoro vertex algebra~\cite{WilliamsVirasoro},
but not, for example, for the simple affine vertex algebras, the irreducible quotients of the universal affine vertex algebras.

This construction of a factorization algebra starting from a vertex algebra was suggested to the author by Andr\'e Henriques, inspired by work of Huang. 
In~\cite{HuangFunctionalAnalyticTheoryI,HuangFunctionalAnalyticTheoryII}, Huang studies locally convex completions of the underlying vector space of a vertex algebra~$\V$.
These locally convex completions are algebras over the~$E_2$-operad of little discs and their multiplication maps are related to the vertex operators of~$\V$.

A factorization algebra~$F$ on~$\C$ assigns a vector space~$F(U)$ to every open subset~$U \subseteq \C$ and extension maps~$F(U) \rightarrow F(V)$ for inclusions~$U\subseteq V$ of open subsets.
These extension maps assemble into a precosheaf.
It is part of the definition of a factorization algebra that this precosheaf is a cosheaf for certain open covers called \emph{Weiss} covers. 
Furthermore, a factorization algebra on~$\C$ has isomorphisms
\[
  F(U) \otimes F(V) \cong F(U \sqcup V) 
\]
for $U,V$ disjoint open subsets of~$\C$.
So far, we have described factorization algebras on~$\C$ with values in the symmetric monoidal category of vector spaces.
We have found it convenient to consider factorization algebras with values in the symmetric monoidal category of complete bornological vector spaces with the symmetric monoidal product given by the completed tensor product of bornological spaces. 

\begin{example*} Let~$\mathfrak{g}$ be a finite-dimensional Lie algebra over~$\C$ and~$\kappa$ a symmetric invariant form on~$\mathfrak{g}$.
Let~$V_{\mathfrak{g},\kappa}$ denote the corresponding universal affine vertex algebra. 
We do not known if~$\mathbf{F}V_{\mathfrak{g},\kappa}$ is isomorphic to the universal affine factorization algebra~$F_{\mathfrak{g}, \kappa}$ from~\cite{CostelloGwilliam}, called the Kac-Moody factorization algebra there and denoted~$\mathcal{F}^\kappa$. 
These factorization algebras take values in the category of chain complexes of differentiable vector spaces. 
However, on open subsets of~$\C$, as opposed to arbitrary Riemann surfaces,~$\mathcal{F}^\kappa$ is concentrated in degree zero by the arguments of~\cite{CostelloGwilliam}. 
We conjecture that~$\mathbf{F}V_{\mathfrak{g},\kappa} \simeq F_{\mathfrak{g},\kappa}$ as prefactorization algebras of differentiable chain complexes, meaning that there is a quasi-isomorphism between these two prefactorization algebras.
Alternatively, we may form the zeroth homology of their construction as a complete bornological space and hope for an isomorphism of factorization algebras with values in the category of bornological spaces. 
\end{example*}

This article is based on the author's thesis which has three parts. 
The first part is available as~\cite{GeometricVertexAlgebras} and not included in this article.
The first and second section of this article are based on the second and third part.  
The introductions of these two sections contain further remarks on other people's related work. 

In~\cite{GeometricVertexAlgebras}, we recall the definitions of~$\Z$-graded vertex algebras and geometric vertex algebras and give a self-contained account of their equivalence. 
This is a simpler analogue of a theorem of Huang which also incorporates infinitesimal conformal symmetries, see~\cite{HuangTwoDimConfGeomAndVOAs}.
Given a~$\Z$-graded vector space, the set of vertex algebra structures on it is in bijection with the set of geometric vertex algebra structures.
Geometric vertex algebras are more similar to factorization algebras than vertex algebras, and our construction of a factorization algebra from a~$\Z$-graded vertex algebra is phrased entirely in terms of its geometric vertex algebra.
A geometric vertex algebra consists of a~$\Z$-graded vector space~$\V = \bigoplus_{k \in \Z} V_k$ together with a sequence of maps~$\V^{\otimes m} \rightarrow \mathcal{O}(\C^m\setminus \Delta; \Vbar)$ for~$m \geq 0$.
All of these maps are called~$\mu$, and their values in~$\Vbar := \prod_{k\in \V} \V_k$ are denoted by
\[
  \mu(a)(z) = \mu(a,z) = \mu(a_1,z_1,\ldots, a_m, z_m)
\]
for~$a \in \V^{\otimes m}$,~$z \in \C^m \setminus \Delta$, and~$a = a_1\otimes \ldots \otimes a_m$.
Here,~$\C^m \setminus \Delta$ is the subset of~$\C^m$ consisting of~$m$-tuples~$(z_1,\ldots,z_m)$ with~$z_i\neq z_j$ for~$i\neq j$, and~$\mathcal{O}(\C^m\setminus \Delta; \Vbar)$ is the set of holomorphic functions in the sense that, after projecting to each component of~$\Vbar$, these take values in a finite-dimensional subspace and are holomorphic. 
The relationship of these~$m$-ary operations to the vertex operators of the corresponding vertex algebra is 
\[
  \mu(a_1,z_1,\ldots,a_m,z_m) = Y(a_1,z_1)\ldots Y(a_m,z_m)\vac
\]
for~$a_1,\ldots, a_n \in \V$ and~$z \in \C^m$ with~$|z_1| > \ldots > |z_m|$.
We refer the reader to~\cite{GeometricVertexAlgebras} for more details. In particular, there we describe properties of a geometric vertex algebra~$(\V,\mu)$ like associativity. 

In Section~\ref{section:MeromorphicPrefactorizationAlgebras}, we define holomorphic prefactorization algebras taking values in the category of complete bornological vector spaces.
To do so, we first summarize various basic facts about complete bornological vector spaces. 
We then describe the constructions~$\mathbf{F}$ and~$\mathbf{V}$ going between geometric vertex algebras and holomorphic prefactorization algebras with discrete weight spaces and meromorphic operator product expansion. 
We check that~$\mathbf{F}\V$ is in fact such a prefactorization algebra and that~$\mathbf{V}\mathbf{F}\V \cong \V$ for every geometric vertex algebra~$\V$.

In Section~\ref{section:FactorizationAlgebras}, we recall the definition of a factorization algebra and prove that~$\mathbf{F}\V$ is a factorization algebra.
\paragraph{Acknowledgments.} 
I would like to thank my advisors Peter Teichner and Andr\'e Henriques for their advice, support, suggestions, and encouragement.
Thank you, Andr\'e, for the very large number of online conversations we had over the last years.
Furthermore, thanks go to
Bertram Arnold,
Christian Blohmann,
Kevin Costello,
Josua Faller,
Vassili Gorbounov,
Owen Gwilliam,
Malte Lackmann,
Jack Kelly,
Achim Krause,
Mathoverflow user user48958,
Mathoverflow user user131781,
Eugene Rabinovich,
Ingo Runkel,
Claudia Scheimbauer,
Jan Steinebrunner,
Christoph Weis, 
Katrin Wendland, 
Jochen Wengenroth,
Brian Williams,
Mahmoud Zeinalian,
and
Tomas Zeman
for helpful conversations or correspondence.
Thanks to Andr\'e Henriques and Eugene Rabinovich for their feedback on a draft. 
Thanks to Prof.\,Stefan Schwede and Prof.\,Michael K\"ohl for agreeing to be on my thesis committee.
During my studies, I visited UC Berkeley and the University of Oxford, and I would like to thank my advisors, these institutions, and their staff for making this possible. 
Most of this work was carried out at the Max Planck Institute for Mathematics in Bonn and financially supported by its IMPRS on Moduli Spaces, and I would like to thank MPIM and its wonderful staff.
I thank my friends and family, and in particular my parents, for their support and encouragement.

%% file: mero_prefa.tex
\Section{Meromorphic Prefactorization Algebras}\label{section:MeromorphicPrefactorizationAlgebras}
A prefactorization algebra~$F$ on a space~$X$ assigns a vector space~$F(U)$ to every open~$U\subseteq X$ and has multiplication maps
\[
F(U) \otimes F(V) \longrightarrow F(W)
\]
for all open subsets~$U\sqcup V \subseteq W \subseteq X$. 
These maps are required to be associative, unital and symmetric in an appropriate sense.
Prefactorization algebras have all of the data of a factorization algebra but not necessarily their local-to-global properties, which we treat in Section~\ref{section:FactorizationAlgebras}.
In this section, we introduce the additional structures and properties of a prefactorization algebra on~$X=\C$ which allow us to obtain a geometric vertex algebra~$\mathbf{V}F$ from it. 
Roughly speaking, the underlying~$\Z$-graded vector space of~$\mathbf{V}F$ consists of the weight spaces for an action of~$\D^\times = \{ z \in \C \mid 0 < |z|< 1\}$ on a small disc around zero, and the multiplication map~$\mu : \mathbf{V}F\otimes \mathbf{V}F \rightarrow \overline{\mathbf{V}F}$ for~$(z,w) \in \C^2 \setminus \Delta$ corresponds to the multiplication map of~$F$ for the inclusion of two disjoint small discs centered at~$z$ and~$w$ into a large disc centered at zero. 
Here, the additional structure of~$F$ being affine-linearly invariant provides some of the isomorphisms between the~$F(d)$ for the different discs~$d \subseteq \C$.
However, we also make use of the maps induced by inclusions to compare the values assigned to concentric discs of different sizes.

The first three subsections contain the relevant definitions and the verification of the axioms of the geometric vertex algebra~$\mathbf{V}F$ associated with a holomorphic prefactorization algebra on~$\C$ with discrete weight spaces and meromorphic operator product expansion (OPE).
In the fourth subsection, we take a geometric vertex algebra~$\V$ and describe~$\mathbf{F}\V$ as a holomorphic affine-linearly invariant prefactorization algebra on~$\C$. 
The fifth subsection establishes that~$\mathbf{F}\V$ has discrete weight spaces and meromorphic operator product expansion, 
and that~$\mathbf{V}\mathbf{F}\V \cong \V$ as geometric vertex algebras. 
\RelatedWork{} 
Huang~\cite{HuangFunctionalAnalyticTheoryI} considers a locally convex completion~$H$ of a finitely generated degreewise finite-dimensional bounded-below~$\Z$-graded vertex algebra~$\V$.
The construction of this locally convex completion uses the standard disc in~$\C$.  
He constructs maps~$H \otimes H \rightarrow H$ for pairs-of-pants and shows that they encode the vertex operators of~$\V$. 
Andr\'e Henriques suggested a variation of this construction for discs to the author, and how to extend it to open subsets of~$\C$. 
Bornological vector spaces appear in the book~\cite{CostelloGwilliam} by Costello and Gwilliam, too. 
There, the bornologies of bornological vector spaces are assumed to arise from a locally convex Hausdorff topology. 
We have split their procedure to obtain vertex algebras from suitable factorization algebras into two steps by making a stop at geometric vertex algebras.
\Subsection{Functional-Analytic Preliminaries}\label{subsection:FunctionalAnalyticPrelim}
Unless otherwise noted, all vector spaces are over~$\C$. 
Recall that a linear map between semi-normed spaces is continuous if and only if it is bounded.
Continuity can be phrased in terms of the collections of open subsets.
Boundedness can be phrased in terms of the collections of bounded subsets: a map is bounded if and only if it sends bounded subsets to bounded subsets.
The collection of bounded subsets of a semi-normed space form a \emph{bornology}.
These two perspectives in terms of topologies and bornologies lead to two generalizations of Banach spaces, namely complete locally convex topological vector spaces and complete convex bornological vector spaces, or complete bornological vector spaces for short.
Our factorization algebras take values in the symmetric monoidal category of complete bornological vector spaces.
A complete bornological vector space may be thought of as ascending union of Banach spaces.
The author has made use of~\cite{Meyer,MeyerBook,HogbeNlend,HogbeNlendMoscatelli,ProsmansSchneiders,CostelloGwilliam} to learn about bornological spaces. 
\begin{definition}
  Let~$X$ be a set. A \emph{bornology}~$\mathcal{B}_X$ on~$X$ is a set of subsets of~$X$ such that
  \begin{itemize}
  \item $A\subseteq B \in \mathcal{B}_X \Rightarrow A \in \mathcal{B}_X$,
  \item $A,B \in \mathcal{B}_X \Rightarrow A \cup B \in \mathcal{B}_X$, and
  \item $x \in X \Rightarrow \{ x\} \in \mathcal{B}_X$.
  \end{itemize}
  We call the elements of~$\mathcal{B}_X$ the \emph{bounded} subsets of~$X$.
  A map~$X \rightarrow Y$ of sets with bornologies~$\mathcal{B}_X$ and~$\mathcal{B}_Y$ is \emph{bounded} if~$f(\mathcal{B}_X) \subseteq \mathcal{B}_Y$.
  Now, let~$X$ be a vector space over~$\C$.
  A subset~$A$ of~$X$ is called \emph{absolutely convex} if~$\lambda x + \mu y \in A$ for all~$x,y\in A$ and~$\lambda, \mu \in \C$ with~$|\lambda| + |\mu| \leq 1$.
  We let~$\langle A \rangle$ denote the absolutely convex hull of~$A$, that is, the smallest absolutely convex subset of~$X$ containing~$A$.
  A \emph{bornological vector space}~$X$ is a vector space~$X$ together with a bornology such that
  \begin{itemize}
    \item $B \in \mathcal{B}_X, \lambda > 0 \Rightarrow \lambda B \in \mathcal{B}_X$ and
    \item $B \in \mathcal{B}_X \Rightarrow \langle B \rangle  \in \mathcal{B}_X$.
  \end{itemize}
\end{definition}
If~$X$ is a bornological vector space, and~$A,B \in \mathcal{B}_X$, then~$A+B \in \mathcal{B}_X$, because $A+B \subseteq \langle 2(A\cup B)\rangle$.
\begin{definition}
  Let~$\BVS$ denote the category of bornological vector spaces and bounded linear maps.
\end{definition}
The category~$\BVS$ is additive but not abelian and has all colimits and limits.
The direct sum~$\bigoplus_{i\in I} X_i$ of a family of bornological vector spaces is the direct sum of the underlying vector spaces together with the following bornology. 
A subset~$B \subseteq \bigoplus_{i\in I} X_i$ is bounded if there exists a family of bounded subsets~$C_i \subseteq X_i$ for~$i\in I$ with~$B \subseteq \bigoplus_{i\in I} C_i$ such that all except finitely many~$C_i$ are zero.
This bornology is the unique bornology turning the direct sum into the coproduct in~$\BVS$. 

If~$X$ is a sub vector space of a bornological vector space~$Y$, then~$Y/X$ is a bornological vector space where a subset is defined to bounded if it is the image of a bounded subset under the quotient map~$q: Y \rightarrow Y/X$:
\[
\mathcal{B}_{Y/X} := \{ C \subseteq Y/X \mid \exists B \in \mathcal{B}_Y : C=q(B) \}
\]
This bornology is called the \emph{quotient bornology}.
If~$f:X \rightarrow Y$ is a bounded linear map, then~$Y/f(X)$ together with the quotient map from~$Y$ to~$X$ is a cokernel of~$f$ in the category of bornological vector spaces.

The product~$\prod_{i\in I} X_i$ in the category of bornological vector spaces has the product of vector spaces as its underlying vector space. 
Its bornology consists of those sets which are contained in a set of the form~$\prod_{i\in I} B_i$ for bounded subsets~$B_i \subseteq X_i$ for~$i\in I$.
A sub vector space~$X$ of a bornological vector space~$Y$ is a bornological vector space by defining subsets to be bounded if they are bounded in~$Y$,
\[
\mathcal{B}_X = \{ A\subseteq X \mid A \in \mathcal{B}_Y  \}\displayperiod
\]
This bornology is called the \emph{subspace bornology}.
If~$f:X \rightarrow Y$ is a bounded linear map, then~$f^{-1}(0)$ with the subspace bornology is a kernel of~$f$ in the category of bornological vector spaces.
Our next goal is the definition of completeness for bornological spaces.
\begin{definition}
  If~$X$ is a semi-normed space, then the set of all subsets of~$X$ bounded w.\,r.\,t.\ its norm defines a bornology on~$X$ and turns~$X$ into a bornological space.
This defines a full embedding of the category of semi-normed spaces and bounded maps into~$\BVS$.
The image of this functor consists of those bornological spaces whose bornology arises from some semi-norm and we call such bornological spaces \emph{semi-normable}.
Similarly, we call such bornological spaces \emph{normable} if their bornology comes from a norm and \emph{completely normable}, or \emph{Banachable}, if their bornology comes from a complete norm.
\end{definition}
Two semi-norms on a vector space are equivalent if and only if their bornologies are equal.
\begin{definition}
 Let~$X$ be a bornological vector space with bornology~$\mathcal{B}_X$.
 A bornological vector space~$Y$ with bornology~$\mathcal{B}_Y$ is a \emph{subobject} of~$X$ if~$Y\subseteq X$ is a sub vector space and~$\mathcal{B}_Y \subseteq \mathcal{B}_X$.
\end{definition}
The inclusion~$Y \subseteq X$ of a subobject is bounded, but~$Y$ does not necessarily carry the subspace bornology inside~$X$.
\begin{definition}
  A bornological space is called \emph{complete}, if every bounded subset is contained in some completely normable subobject and bounded there. 
   A \emph{completion} of~$X$ is an initial object in the full subcategory of~$\BVS_{X/}$ whose objects are maps from~$X$ to a complete bornological vector space.
\end{definition}
This definition of completeness of bornological spaces is easily seen to be equivalent to the one given in~\cite[3:2.1]{HogbeNlend} using completant discs.  
If a map~$f:X \rightarrow \overline{X}$ is a completion of~$X$, we usually omit the map~$f$ and call~$\overline{X}$ a completion of~$X$.
If~$X$ is already complete, then~$\id_X$ is a completion of~$X$.
Completions are unique up to unique isomorphism as initial objects in the category mentioned above.
The category of complete bornological vector spaces has all colimits and limits, which are described further below. 
Completions exist because the completion of a semi-normable space is its usual completion and
\[
  \overline{X} = \colim_{\stackrel{Y\subseteq X}{\text{semi-norm.}}} \overline{Y}
\]
is a completion of~$X$ where the colimit is taken inside the category of complete bornological vector spaces. 

The direct sum of a family of complete bornological vector spaces is complete.
In particular, such a sum is the coproduct in the category of complete bornological vector spaces. 
The analogous statements hold for products.
We now define the topology of sequentially closed sets, or b-closed sets, in order to describe cokernels in the category of complete bornological spaces. 
\begin{definition}
  A sequence~$(x_n)_n$ in a bornological vector space~$X$ \emph{converges} to~$y \in X$ if there is a semi-normable subobject~$Y \subseteq X$ containing~$y$ and all~$x_n$ such that~$(x_n)_n$ converges to~$y$ in~$Y$.
\end{definition}
Convergence of sequences in a semi-normable bornological space~$Y$ is independent of the choice of a semi-norm inducing the bornology of~$Y$.
If~$X$ is bornological vector space, then the set of sequentially closed sets form the closed sets of a topology on~$X$.
These closed sets are called \emph{b-closed}.
Bounded linear maps are continuous w.\,r.\,t.\ this topology.
A subspace of a complete bornological space is b-closed if and only if it is complete, see~\cite[3:2.3 Proposition 1]{HogbeNlend}. 
If~$X$ is a b-closed subspace of a complete bornological space~$Y$, then~$Y/X$ is complete, see~\cite[3:2.3 Proposition 2]{HogbeNlend}.  
Therefore, if~$f:X\rightarrow Y$ is a bounded linear map, then~$Y/\overline{f(X)}$ is a cokernel of~$f$ in the category of complete bornological vector spaces.
The kernel of a bounded linear map~$f:X \rightarrow Y$ between complete spaces in the category of complete bornological vector spaces is again~$f^{-1}(0)$, since bounded linear maps are continuous w.\,r.\,t.\ the topology of b-closed sets and~$\{0\}$ is b-closed in every complete bornological space.
The quotient of a complete bornological space by the action of a finite group is again complete, as can be seen by identifying it with the complete bornological space of invariants.

\begin{definition}\label{definition:Continuity}
  Let~$U\subseteq \R^k$ and~$X$ be a complete bornological space. A function~$f:U \rightarrow X$ is \emph{continuous} if~$\lim _n x_n = x_\infty$ implies~$\lim_n f(x_n) = f(x_\infty)$.
\end{definition}
Continuity in the sense of Definition~\ref{definition:Continuity} is equivalent to the property that every compact~$K\subseteq U$ is mapped into a completely normable subobject~$Y$ of~$X$ for which~$f|_K$ is continuous as a map to~$Y$ in the usual sense.  
\begin{definition}
  Let~$U$ be an open subset of~$\C^n$ and let~$X$ be a complete bornological vector space.
  A map~$f:U \rightarrow X$ of sets is called \emph{holomorphic} if every~$p\in U$ has an open neighborhood~$V$ s.\,t.~$f(V) \subseteq Y$ for some completely normable subobject~$Y \subseteq X$ and the map~$f|_V: V \rightarrow Y$ is holomorphic.
\end{definition}
\begin{definition}
  Let~$X$ be a locally convex topological vector space.
  Let~$X'$ be its continuous \emph{dual}, that is, the vector space of continuous linear functionals on~$X$.
  Let~$B\subseteq X'$ be bounded if and only if there exists a continuous semi-norm~$p$ on~$X$ with~$|\alpha(x)| \leq p(x)$ for all~$\alpha \in B$ and~$x \in X$. This turns~$X'$ into a bornological vector space.
\end{definition}
If~$X$ is a locally convex space, then~$X'$ is always complete as a bornological vector space.
Our main example of a non-discrete bornological vector space is the complete bornological vector space~$\mathcal{O}'(U) := \mathcal{O}(U)'$ of analytic functionals on an open subset~$U\subseteq \C^k$.
Here,~$\mathcal{O}(U)$ is the locally convex space of holomorphic functions on~$U$ whose topology is defined by the supremum semi-norms~$||\blank||_{\infty,K}$ on compact subsets of~$K$.
This means that a subset~$B \subseteq \mathcal{O}'(U)$ is bounded if and only if there exists a compact subset~$K \subseteq U$ and a number~$C >0$ s.\,t.
\[
  |\alpha(f)| \leq C ||f||_K
\]
for all~$\alpha \in B$ and~$f \in \mathcal{O}(U)$.
\begin{definition}
  If~$U \subseteq \C^n$ is open and~$K \subseteq U$ is compact, then~$\mathcal{O}'_K(U)$ is the subset of analytic functionals bounded w.\,r.\,t.~$||\blank||_{\infty,K}$, the supremum semi-norm on~$K$. 
\end{definition}
Note that~$\mathcal{O}'_K(U)$ is a Banach space and a completely normable subobject of~$\mathcal{O}'(U)$.
Another description of the complete bornological space~$\mathcal{O}'(U)$ is
\[
  \mathcal{O}'(U) = \bigcup_{\underset{\text{compact}}{K \subseteq U}} \mathcal{O}'_K(U)
\]
where the bornology on~$\mathcal{O}'(U)$ is the union of the bornologies on the~$\mathcal{O}'_K(U)$, that is, a subset~$B \subseteq \mathcal{O}'(U)$ is bounded if and only if there is a compact~$K \subseteq U$ such that~$B$ is a bounded subset of~$\mathcal{O}'_K(U)$.  
\begin{definition}\label{definition:BornologicalTensorProduct}
Given bornological vector spaces~$X$ and~$Y$, we define the bornological tensor product~$X\otimes Y$ to be the algebraic tensor product~$X\otimes Y$ with the smallest convex bornology for which all sets of the form~$B\otimes C$ for~$B\subseteq X$ and~$C \subseteq Y$ are bounded. 
\end{definition}
\begin{definition}\label{definition:CompletedTensorProduct}
  For bornological vector spaces~$X$ and~$Y$, we define the \emph{completed tensor product}~$X \barotimes Y$ to be the completion of~$X \otimes Y$.  
\end{definition}  
Bornological vector spaces form a symmetric monoidal category~$\BVS$ whose structure maps are obtained from the symmetric monoidal category of vector spaces by forgetting the bornology.
It is clear that these maps are bounded.
They induce the structure of a symmetric monoidal category on the category~$\CBVS$ of complete bornological spaces via the completion functor. 
The completed tensor product commutes with colimits in each variable separately.
We take the liberty to denote the image of~$x \otimes y$ in~$X\barotimes Y$ by~$x \otimes y$, too. 

Every finite-dimensional vector space has a unique bornology turning it into a complete bornological vector space.
If~$X$ is a vector space, then we can turn it into a bornological vector space by declaring a subset to be bounded if it is contained in a finite-dimensional subspace and is bounded there. 
These bornological vector spaces are called \emph{discrete}.
A bornological vector space is discrete if and only if it is a direct sum of one-dimensional bornological vector spaces.
Discrete bornological vector spaces are complete.
Taking the tensor product with a discrete bornological vector space preserves completeness.
One way to see this is to use the fact that the tensor product is left adjoint to the internal hom in bornological spaces, defined next, and thus preserves colimits, and in particular direct sums.
\begin{definition}\label{definition:BornologyOnSetsOfBoundedMaps}
  Let~$Y$ and~$Z$ be bornological vector spaces.
  A subset~$B\subseteq \BVS(Y,Z)$ is bounded if, for every bounded~$C\subseteq Y$ there exists a bounded~$D\subseteq Z$ such that~$f(C) \subseteq D$ for all~$f \in B$.
  The collection of these bounded sets turns~$\BVS(Y,Z)$ into a bornological vector space. 
\end{definition}
Both~$\CBVS$ and~$\BVS$ are closed in the sense of having internal homs adjoint to the tensor product.
If~$Z$ is complete in the above definition, then~$\BVS(Y,Z)$ is complete, so the internal hom in~$\CBVS$ agrees with that in $\BVS$.

Since the passage from factorization algebras on~$\C$ to~$\Z$-graded geometric vertex algebras is via the direct sum of the weight spaces of the value of the factorization algebra on a disc, we now consider~$\Z$-graded complete bornological spaces. 
\begin{definition}
  A~$\Z$-graded complete bornological vector space~$X$ is a complete bornological vector space~$X$ whose underlying vector space
  is equipped with a~$\Z$-grading such that~$X$ is the direct sum of the~$X_k$ for~$k\in \Z$ in the category of complete bornological spaces.
\end{definition}
A~$\Z$-graded complete bornological space~$X$ is the same thing as a~$\Z$-graded vector space each of whose graded components~$X_k$ is additionally equipped with a bornology which turns~$X_k$ into a complete bornological space. 
\begin{proposition}\label{proposition:HolomorphicIfAndOnlyIfComponentsAre}
Let~$(X_i)_{i\in I}$ be a family of complete bornological vector spaces. 
Let~$U\subseteq \C^n$ be an open subset and let~$f : U \rightarrow \prod_{i\in I} X_i$ be a map of sets.
Then~$f$ is holomorphic if and only if each of its components is.
\end{proposition}
\begin{proof}
A map~$g$ into a complete bornological space is called holomorphic if it is locally a holomorphic map to some completely normable subobject.
Therefore,~$g$ holomorphically maps the interior of every compact subset~$K\subseteq U$ to some completely normable subobject depending on~$K$. 
Let each component~$f_i$ be holomorphic.
Without loss of generality,~$0 \in U$ and we show that~$f$ is holomorphic on a neighborhood of~$0$. 
Let~$r>0$ with~$K:=\overline{B_r(0)}\subseteq U$. 
For all~$i\in I$, there is a completely normable subobject~$Y_i$ of~$X_i$ such that~$f_i(K) \subseteq Y_i$ and~$f_i$ restricted to interior of~$K$ is holomorphic as a map with target~$Y_i$.
For all~$i \in I$, the Taylor series~$\sum^\infty_{k = 0} f_{i,k} z^k$ of~$f_i$ converges absolutely at~$r$ in~$Y_i$. 
For each~$i \in I$, pick a norm~$p_i$ on~$Y_i$ so that
\[
  \sum^\infty_{k = 0} p_i(f_{i,k}) r^k \leq 1\displayperiod
\]
Such a norm exists because absolute convergence means that there is some norm for which this sum is finite.
Let~$Y$ be the completely normable subobject of~$\prod_i X_i$ consisting of those~$x$ with finite norm~$p(x) := \sup_i p_i(x_i)$.
Let~$f_k \in \prod_i X_i$ be given by~$(f_k)_i = f_{i,k}$. 
Then~$f_k \in Y$ because
\[
  p(f_k) = \sup_i p_i(f_{i,k}) \leq \frac{1}{r^k}
\] 
so
\begin{align*}
  \sum^\infty_{k=0}  p(f_k) |z|^k  \leq \sum^\infty_{k=0}  \frac{|z|^k}{r^k} < \infty 
\end{align*}
for~$|z| < r$, meaning that~$\sum^\infty_{k=0} f_k z^k$ absolutely converges on~$B_r(0)$. 
Since the projection~$\prod_j X_j \rightarrow X_i$ is bounded for all~$i$ and maps~$f_k$ to~$f_{i,k}$, it follows that this sum converges to~$f(z)$. 

Conversely, if~$f$ is holomorphic, then each component of~$f$ is holomorphic, because the projections~$\prod X \rightarrow Y_i$ are bounded and the composition of a holomorphic map with a bounded linear map to another complete bornological space is again holomorphic.  
\end{proof}
\begin{definition}\label{definition:BornologyOnHolomorphicFunctions}
Let~$X$ be a complete bornological space and let~$U\subseteq \C^n$ be open.
A subset~$B \subseteq \mathcal{O}(U;X)$ is bounded if and only if there exists a map~$\beta: \mathcal{P}_c(U) \rightarrow \mathcal{B}_X$ from the set of compact subsets of~$U$ to the set of bounded subsets of~$X$ such that~$f(K) \subseteq \beta(K)$ for all~$f\in B$ and~$K\subseteq U$ compact. 
With these bounded sets~$\mathcal{O}(U;X)$ is a complete bornological space.
\end{definition}
Let~$X$ and~$Y$ be complete bornological spaces. 
For~$U\subseteq \C^m$ open, the pointwise tensor product of holomorphic functions~$f:U\rightarrow X$ and~$g:U \rightarrow Y$ 
is the holomorphic function~$fg(z) = f(z) \otimes g(z) \in X \barotimes Y$ of~$z \in \C^n$.
This defines a bounded linear map
\[
\mathcal{O}(U;X) \otimes \mathcal{O}(U;Y) \longrightarrow \mathcal{O}(U; X\barotimes Y)\displayperiod
\]
Similarly, we have the external product
\(
  f\times g(z,w) = f(z) \otimes g(w) 
\)
of~$(z,w) \in U\times V$ for~$V \subseteq \C^n$ open and~$g:V \rightarrow Y$.
The external product defines a bounded linear map
\[
\mathcal{O}(U;X) \otimes \mathcal{O}(V;Y) \longrightarrow \mathcal{O}(U\times V; X\barotimes Y)\displayperiod
\]
For~$\alpha \in \mathcal{O}'(U)$ and~$\beta \in \mathcal{O}'(V)$, the external product~$\alpha \times \beta \in \mathcal{O}'(U\times V)$ is defined by 
\begin{align}
  \alpha \times \beta(f) = \alpha( z \mapsto \beta( w\mapsto f(z,w)))\displayperiod  \label{equation:ExternalProductOfAnalyticFunctionals}
\end{align}
This defines a bounded linear map 
\[
  \mathcal{O}'(U) \otimes \mathcal{O}'(V) \rightarrow \mathcal{O}'(U \times V)\displaycomma
\]
called the external product of analytic functionals. 
\Subsection{Holomorphic Affine-Linearly Invariant Prefactorization Algebras}
We define prefactorization algebras on a topological space~$X$ with values in a symmetric monoidal category~$\mathcal{C}$. 
For~$X$ a~$G$-space for a group~$G$, there is also the notion of a~$G$-invariant prefactorization algebra on~$X$. 
Being~$G$-invariant is not just a property, but comes with isomorphisms~$F(U) \cong F(gU)$ for~$g\in G$ and~$U\subseteq X$ open. 
If~$G$ is complex-analytic, we also define what it means for a prefactorization algebra with values in~$\CBVS$ to be holomorphic. 
Our case of interest is~$X=\C$ with the group of affine-linear isomorphisms~$\C^\times \ltimes \C$.
\begin{definition}
Let~$X$ be a topological space and~$(\mathcal{C},\otimes, \mathbf{1}_\mathcal{C})$ a symmetric monoidal category.
A (unital) prefactorization algebra on~$X$ with values in~$\mathcal{C}$ consists of
\begin{itemize}
\item an object~$F(U)$ of~$\mathcal{C}$ for every open~$U\subseteq X$,
\item multiplication maps
\[
  M^{U, V}_W : F(U) \otimes F(V) \longrightarrow F(W)
\]
in~$\mathcal{C}$ for all disjoint open subsets~$U,V$ of some open subset~$W\subseteq X$, and
\item unit maps
  \[
    \mathbf{1}_\mathcal{C} \rightarrow F(\emptyset)
  \] 
\end{itemize}
such that:
\begin{itemize}
\item (associativity) For~$W \subseteq X$ open,~$V_1,V_2 \subseteq W$ open, and~$U_1,U_2,U_3 \subseteq X$ open and pairwise disjoint with~$U_1,U_2 \subseteq V_1$,~$U_2,U_3 \subseteq V_2$, and~$U_1$ disjoint from~$V_2$ and~$V_1$ disjoint from~$U_3$, the diagram
\begin{center}
\begin{tikzpicture}
  \node (left) at (-2.3,2) {$(F(U_1) \otimes F(U_2)) \otimes F(U_3)$};
  \node (right) at (2.3,2) {$F(U_1) \otimes (F(U_2) \otimes F(U_3))$};
  \node (UV) at  (342:3){$F(U_1) \otimes F(V_2)$};
  \node (VU) at (198:3)  {$F(V_1) \otimes F(U_3)$};
  \node (W) at (270:3) {$F(W)$};
  \path (left) to node {$\cong$} (right) ;
  \draw [->] (left) to node[left] {$M^{U_1,U_2}_{V_1} \otimes \id$} (VU);
  \draw [->] (right) to node[right] {$\id \otimes M^{U_2,U_3}_{V_2}$} (UV);
  \draw [->] (VU) to node[below left] {$M^{V_1,U_3}_{W}$} (W);
  \draw [->] (UV) to node[below right] {$M^{U_1,V_2}_{W}$} (W);
\end{tikzpicture}
\end{center}
commutes where the unlabeled isomorphism is the associator in~$\mathcal{C}$.
\item (unitality) For~$U \subseteq X$ open, the composition
  \[
    F(U) \cong F(U) \otimes \mathbf{1}_\mathcal{C} \longrightarrow F(U) \otimes F(\emptyset) \longrightarrow F(U)
  \]
  is the identity of~$F(U)$.
\item (compatibility with braiding) For disjoint, open~$U,V \subseteq X$, the braiding~$\beta_{F(U),F(V)}$ of~$F(U)$ with~$F(V)$ in~$\mathcal{C}$ is compatible with multiplication in the sense that
  \[
    M^{V, U}_{U \sqcup V} \after \beta_{F(U),F(V)} = M^{V,U}_{U\sqcup V}\displayperiod
  \]
\end{itemize}
\end{definition}
\begin{remark}\label{remark:PrefactorizationAlgebrasAndMAryOperations}
Let~$F$ be a prefactorization algebra on a topological space~$X$.
If~$U_1, \ldots, U_m \subseteq V$ are pairwise disjoint open subsets of~$X$, we get an~$m$-ary multiplication map
\[
M^{U_1,\ldots,U_m}_V :  F(U_1) \otimes \ldots \otimes F(U_m) \longrightarrow F(V)
\]
subject to associativity as in~\cite[3.1.1]{CostelloGwilliam}. 
The conditions involving the maps for~$m=0$ express unitality.
Recall that a precosheaf on a topological space~$X$ is a functor from the category of opens of~$X$ and inclusions among them to some other category.
For~$m=1$, the maps 
\[
M^U_V :F(U) \rightarrow F(V)
\] 
for~$U \subseteq V$ open are the extension maps of the underlying precosheaf of~$F$, which we also denote by~$F$.
If~$F$ is a precosheaf, we denote its extension maps by~$M^U_V$ for notational compatibility.
\end{remark}
\begin{definition}
Let~$X$ be a topological space with an action of a group~$G$.
The category of open subsets of~$X$ has a left~$G$-action which is compatible with taking disjoint unions.
Therefore~$G$ acts on the category of prefactorization algebras.
A prefactorization algebra is called~\emph{$G$-invariant} if it is equipped with the data of a fixed point for this action. 
This consists of maps
\[
\sigma^U_g : F(U) \longrightarrow F(gU)
\]
for~$g\in G$ and~$U\subseteq X$ open such that
\begin{itemize}
\item~$\sigma^U_{\id} = \id_{F(U)}$ for~$U\subseteq X$ open,
\item~$\sigma^{hU}_g \after \sigma^{U}_h = \sigma^U_{gh}$ for~$g,h \in G$ and~$U\subseteq X$ open,
\item $\sigma^\emptyset_g: F(\emptyset) \rightarrow F(\emptyset)$ is compatible with the unit map for all~$g \in G$
\item $\sigma^W_g \after M^{U,V}_W = M^{gU,gV}_{gW} \after (\sigma^U_g\otimes \sigma^V_g)$ for~$g\in G$ and~$U,V \subseteq W$ open and disjoint and~$W\subseteq X$ open.
\end{itemize} 
A prefactorization algebra on~$\C$ is called \emph{affine-linearly invariant} if it is invariant w.\,r.\,t.\ the group~$G=\C^\times \ltimes \C$ of affine transformations of~$\C$. 
\end{definition}
\begin{remark}\label{remark:InvariantPrefactorizationAlgebrasAndMAryOperations}
Note that the maps~$\sigma^U_g$ are isomorphisms.
The~$m$-ary versions of~$M$ constructed from the unit and the binary~$M$, which are part of a prefactorization algebra~$F$ by definition,
are also fixed by the action of~$G$ if~$F$ is~$G$-invariant:
\[
  \sigma^V_g \after M^{U_1,\ldots,U_m}_V = M^{{gU_1},\ldots,{gU_m}}_{gV} \after (\sigma^{U_1}_g \otimes \ldots \otimes \sigma^{U_m}_g)
\]
for~$U_1\sqcup \ldots\sqcup U_m \subseteq V$ open subsets of~$X$ and~$g \in G$.
\end{remark}
We make analogous definitions for precosheaves instead of prefactorization algebras.
We denote the extension maps of a precosheaf~$F$ by
\[
M^U_V :F(U) \rightarrow F(V)
\] 
to be compatible with the notation for prefactorization algebras.
The extension maps of a~$G$-invariant precosheaf satisfy
\[
\sigma^V_g \after M^U_V = M^{gU}_{gV} \after \sigma^U_g 
\]
for~$g\in G$ and~$U\subseteq V\subseteq X$ open.
\begin{definition}\label{definition:DtimesActionOnValueOnDiscOfInvariantPrecosheaf}
Let~$F$ be a~$\C^\times$-invariant precosheaf on~$\C$ with values in the category of vector spaces. 
For~$r>0$, the semi-group~$\D^\times = \{ z \in \C \mid 0 < |z| < 1\}$ acts on~$B_r(0)$ and hence~$F(B_r(0))$ by
\labelledMapMapsto{\rho}{\D^\times}{\Hom(F(B_r(0)),F(B_r(0)))}{q}{M^{B_{qr}(0)}_{B_r(0)} \after \sigma^{B_r(0)}_{q}\displayperiod}
For a character~$\chi \in \Hom(\C^\times,\C^\times)$, let~$F(B_r(0))_\chi$ be the weight space for~$\chi$, that is, 
\[
F(B_r(0))_\chi = \{ x \in F(B_r(0)) \mid \forall q \in \D^\times : \rho(q)(x) = \chi(q)x \}\displayperiod
\]
\end{definition} 
Let~$F$ be a~$\C^\times$-invariant precosheaf on~$\C$. 
We write~$F(B_r(0))_k$ for the~$k$-th weight space for~$k\in\Z$ which corresponds to~$\chi(q)=q^k$.
The maps
\[
i^r_R := M^{B_r(0)}_{B_R(0)}: F(B_r(0)) \longrightarrow F(B_R(0))
\]
induced by the inclusions for~$0<r\leq R \leq \infty$ are~$\D^\times$-equivariant, so they induce maps
\[
(i^r_R)_\chi : F(B_r(0))_\chi \longrightarrow F(B_R(0))_\chi 
\]
for all~$\chi \in \Hom(\C^\times, \C^\times)$.
\begin{proposition}\label{proposition:InclusionInducesIsomorphismsOnWeightSpaces}
The map~$(i^r_R)_\chi$ is an isomorphism for all characters~$\chi \in \Hom(\C^\times, \C^\times)$ and~$0<r\leq R< \infty$.
\end{proposition}
\begin{proof}
Let~$q=r/R$.
The map~$(i^r_R)_\chi$ is inverse to~$(\sigma^{B_R(0)}_q)_\chi$ up to a non-zero scalar factor, namely~$\chi(q)$.
\begin{align}
  (i^r_R)_\chi \after (\sigma^{B_R(0)}_{q})_\chi = \chi(q)\id_{F(B_R(0))_\chi}\label{display:InclusionInducesIsomorphismsOnWeightSpaces}
\end{align}
by the definition of~$F(B_R(0))_\chi$.  
To prove
\[ (\sigma^{B_R(0)}_{q})_\chi \after (i^r_R)_\chi = \chi(q)\id_{F(B_r(0))_\chi} \]
we replace~$R$ with~$qR=r$ and~$r$ with~$qr$ in Equation~(\ref{display:InclusionInducesIsomorphismsOnWeightSpaces}) to get
\[
(i^{qr}_{r})_\chi \after (\sigma^{B_r(0)}_{q})_\chi = \chi(q) \id_{F(B_r(0))}\\
\]
whose l.\,h.\,s.\ equals
\[
(\sigma^{B_R(0)}_q)_\chi \after (i^{r}_{R})_\chi
\]
because
\[
i^{qr}_{r} \after \sigma^{B_r(0)}_q = \sigma^{B_R(0)}_q \after i^{r}_{R}
\]
because~$F$ is an invariant precosheaf.
\end{proof}
\begin{definition}
Let~$F$ be a precosheaf on~$\C$.
Let
\[
F(z) = \lim_{r>0} F(B_r(z))
\]
denote the \emph{costalk} of~$F$ at~$z\in \C$.
\end{definition}
We restrict attention to weights~$\chi \in \Hom(\C^\times, \C^\times)$ given by~$\chi(q) = q^k$ for some~$k\in \Z$ because vertex algebras are~$\Z$-graded.
The~$k$-th weight space of the costalk is
\[
F(z)_k = \left( \lim_{r>0} F(B_r(z)) \right)_k \cong \lim_{r>0} F(B_r(z))_k
\]
because limits commute. We suppress this isomorphism in our notation.
By Proposition~\ref{proposition:InclusionInducesIsomorphismsOnWeightSpaces},
\[
F(z)_k \cong F(B_R(z))_k 
\]
for all~$R>0$.
\begin{definition}\label{definition:HolomorphicallyInvariantPrecosheaves}
  Let~$X$ be a topological space with the action of a complex-analytic group~$G$. 
  For~$U,V \subseteq X$, let
  \[
  \mathcal{D}_{U,V} = \{g \in G \mid gU \subseteq V \}\displayperiod
  \]
  A~$G$-invariant prefactorization algebra~$F$ on~$X$ with values in the symmetric monoidal category of complete bornological spaces is called \emph{holomorphic} if the map 
  \labelledMapMapsto{\rho_{U,V}}{\mathcal{D}_{U,V}}{\mathbf{\BVS}(F(U),F(V))}{g}{M^{gU}_V \after \sigma^U_g}
  is holomorphic on the interior of its domain.
  This condition only depends on the underlying~$G$-invariant precosheaf of~$F$ and we define holomorphic~$G$-invariant precosheaves using this condition. We abbreviate~$\rho_{U,V}$ by~$\rho$.   
\end{definition} 
\begin{proposition}\label{lemma:WeightSpaceProjections}
  Let~$X$ be a complete bornological vector space with a holomorphic representation
  \[
    \rho : \D^\times \longrightarrow \End(X) := \BVS(X,X)
  \]
  of the semi-group~$\D^\times$. 
  For~$k \in \Z$ the map
  \labelledMapMapsto{l_k}{X}{X_k}{x}{ \ootpi \oint z^{-k-1} \rho(z)(x) dz}  
  to the~$k$-th weight space of~$\rho$ is a well-defined bounded linear map and a~$\D^\times$-equivariant splitting of the inclusion.
\end{proposition}
\begin{proof}
  The contour integral may be taken over any of the circles~$tS^1$ in~$\D^\times$ with~$0<t<1$.
  For~$q \in \D^\times$,
  \begin{align*}
  \rho(q)(l_k(x))&= \ootpi \int_{tS^1} z^{-k-1} \rho(q)(\rho(z)(x)) dz\\
  &= \ootpi \int_{tS^1} z^{-k-1} \rho(qz)(x) dz\\
  &= \ootpi \int_{qtS^1} q^{k+1} z^{-k-1} \rho(z)(x) q^{-1}dz\\
  &= q^k \ootpi \int_{qtS^1} z^{-k-1} \rho(z)(x) dz\\
  &= q^k \ootpi \int_{tS^1} z^{-k-1} \rho(z)(x)dz \quad\text{(deform)}\\ 
  &= q^k l_k(x)\displaycomma
  \end{align*}
  so~$l_k(x) \in X_k$ for~$x \in X$.
  Equivariance uses the commutativity of~$\D^\times$:
  \begin{align*}
  \rho(q)(l_k(x)) &= \ootpi \oint z^{-k-1} \rho(qz)(x) dz\\
  &= \ootpi \oint z^{-k-1} \rho(z)(\rho(q)(x)) dz \\
  &= l_k(\rho(q)(x))
  \end{align*}
  For~$x \in X_k$
  \begin{align*}
  l_k(x) &= \ootpi \oint z^{-k-1} z^kx dz = x \displaycomma
  \end{align*}
  so~$l_k$ is a splitting of the inclusion of~$X_k$ into~$X$.
  
  It remains to prove that~$l_k$ is bounded. 
  Let~$Y \subseteq X$ be a completely normable subobject of~$X$. 
  It suffices to prove that~$l_k|_Y$ is bounded. 
  By the compactness of~$tS^1$ and the holomorphicity of~$\rho$, there is a completely normable subspace~$Z$ of the endomorphisms of~$X$ with~$\rho(tS^1) \subseteq Z$ bounded. 
  If~$B\subseteq Y$ is bounded, then~$\rho(tS^1)(B)$ is bounded inside some completely normable subobject~$Y'$ of~$X$, so the boundedness of~$l_k(B)$ follows, since~$l_k(B)$ is a subset of the closure inside~$Y'$ of the convex hull of~$\rho(tS^1)(B)$.
\end{proof}
Let~$F$ be a holomorphic~$\C^\times$-invariant precosheaf on~$\C$.
For~$r >0$ and~$k\in \Z$, let
\[
l^r_k : F(B_r(0)) \longrightarrow F(B_r(0))_k
\]
denote the weight space projection~$l_k$ for the~$\D^\times$-representation~$F(B_r(0))$.
The costalk~$F(0)$ is a holomorphic~$\D^\times$-representation, too, as the limit of holomorphic representations.
For~$r > 0$, let
\[
\pi_r : F(0) \rightarrow F(B_r(0))
\]
denote the projection to the~$r$-component of the costalk of~$F$ at zero, that is, one of maps exhibiting~$F(0)$ as the limit of the~$F(B_r(0))$.
Of course,~$\pi_r$ is equivariant by construction and hence compatible with the weight space projections, i.\,e.
\begin{ctikzcd}
  F(0) \rar{l_k} \dar{\pi_r} & F(0)_k \dar{\pi_r} \\
  F(B_r(0)) \rar{l^r_k} & F(B_r(0))_k
\end{ctikzcd}
commutes.
It can be shown that the~$\D^\times$-representation on~$F(0)$ extends to a holomorphic~$\C^\times$-representation and its weight space projections~$l_k$ sum to the identity, $k \in \Z$. 
However, to prove associativity of the geometric vertex algebra~$\mathbf{V}F$ associated with certain holomorphic prefactorization algebras~$F$ as in the next subsection, we need to sum up all weight space projections for~$F$ evaluated on a disc of medium size after including into a slightly larger disc.
\begin{proposition}\label{proposition:SumOfWeightSpaceProjections}
  Let~$F$ be a holomorphic~$\C^\times$-invariant precosheaf on~$\C$.
  For~$R>r>0$, we have
  \[
  i^r_R = \sum_{k\in \Z} i^r_R \after l^r_k
  \]
  in the space~$\BVS(F(B_r(0)),F(B_R(0)))$ with absolute convergence in some completely normable subobject. 
\end{proposition}
\begin{proof}
  Let~$A = B_{R/r}(0) - 0$.
  The function 
  \labelledMapMapsto{\rho}{A}{\BVS(F(B_r(0)),F(B_R(0)))}{q}{M^{B_{qr}(0)}_{B_R(0)} \after \sigma^{B_r(0)}_q = i^{qr}_R \after \sigma^{B_r(0)}_q}
  is holomorphic by assumption.  
The Laurent expansion of~$\rho$ can be computed on~$\D^\times$ and is ($0<t<1$)
  \[
    \rho(q) = \sum_{k\in \Z} \ootpi \int_{tS^1} z^{-k-1} \rho(z) dz q^k = \sum_{k\in \Z} (i^r_R \after l^r_k) q^k
  \]
  because
  \begin{align*}
    \ootpi \int_{tS^1} z^{-k-1} \rho(z) dz  &= \ootpi \int_{tS^1} z^{-k-1} i^{zr}_R \after \sigma^{B_r(0)}_z dz \\
    &= \ootpi \int_{tS^1} z^{-k-1} i^r_R \after i^{zr}_r \after \sigma^{B_r(0)}_z dz \\
    &= i^r_R \after \left(\ootpi \int_{tS^1} z^{-k-1} i^{zr}_r \after \sigma^{B_r(0)}_z dz \right)\\
    &= i^r_R \after l^r_k\displayperiod
  \end{align*}
   We evaluate at~$q=1$ to get
  \[
    i^r_R = \sum_{k\in \Z} i^r_R \after l^r_k 
  \]
  with absolute convergence in some completely normable subobject of
  \[
    \BVS(F(B_r(0)),F(B_R(0)))
  \] 
  because the Laurent expansion of a holomorphic function with values in a complete bornological space converges absolutely in some completely normable subobject on any compact subset. 
\end{proof}
\Subsection{From Prefactorization Algebras to Geometric Vertex Algebras}\label{subsection:FromPrefactorizationAlgebrasToGeometricVertexAlgebras}
We fix a holomorphic prefactorization algebra~$F$ and construct a geometric vertex algebra~$\mathbf{F}\V$ from it under additional assumptions which we now describe. 
Let~$\mathbf{V}F$ denote the~$\Z$-graded vector space
\[
\mathbf{V}F = \bigoplus_{k\in \Z} F(0)_k,
\]
the direct sum of the weight spaces of the~$\D^\times$-action on the costalk of~$F$ at zero.
The homogeneous components~$F(0)_k$ are complete bornological vector spaces and we assume these to be discrete, as is done in the setting of differentiable vector spaces in~\cite{CostelloGwilliam} for the weight spaces of~$F$ evaluated on a disc.
(In common examples, the homogeneous components are in fact finite-dimensional and~$\mathbf{V}F$ is bounded from below or even concentrated in non-negative degrees.)
Assuming that the weight spaces of the costalk are discrete, we construct a~$\Z$-graded 
vector space~$\mathbf{V}F$ from~$F$ together with the underlying data of a geometric vertex algebra. 
We prove that these data satisfy the axioms of a geometric vertex algebra except possibly meromorphicity. 
We then say that a holomorphic prefactorization algebra has \emph{meromorphic OPE} if~$\mathbf{V}F$ satisfies the meromorphicity axiom, so that holomorphic prefactorization algebras with discrete weight spaces and meromorphic OPE yield geometric vertex algebras. 
If~$\mathbf{V}F$ is bounded from below, then it is automatically meromorphic by Proposition~2.3 
of~\cite{GeometricVertexAlgebras}.
The author expects that the discreteness assumption can be removed with more work.

The multiplication maps of~$F$ induce a map
\[
M^{z_1,\ldots,z_m}: F(z_1) \otimes \ldots \otimes F(z_m) \longrightarrow \colim_{R>0} F(B_R(0))
\]
for~$z\in \C^m\setminus \Delta$ with~$|z_i| < R$. We get a bounded linear map
\[
\widetilde{M}^{z_1,\ldots,z_m}: F(0)^{\otimes m} \longrightarrow \colim_{R>0} F(B_R(0)) 
\]
by using the translation invariance of~$F$.
This is almost the multiplication map of the geometric vertex algebra associated with~$F$; it remains to precompose with a map from~$\mathbf{V}F$ and postcompose with a map to~$\overline{\mathbf{V}F}$.
The target of~$\widetilde{M}^{z}$ maps to
\[
\overline{\mathbf{V}F} := \prod_{k\in \Z} F(0)_k
\]
as follows. For~$R>0$ and~$k\in \Z$, the map
\(
F(B_R(0)) \longrightarrow F(0)_k
\)
is the composite
\begin{align}
  F(B_R(0)) \overset{l^R_k}{\longrightarrow} F(B_R(0))_k \overset{\cong}{\longrightarrow} \lim_{r> 0} F(B_r(0))_k \cong F(0)_k\displaycomma \label{display:FBRZeroToCostalkWeightSpace}
\end{align}
where the first isomorphism arises from Proposition~\ref{proposition:InclusionInducesIsomorphismsOnWeightSpaces} and the second from the fact that limits commute with each other.
The source of~$\widetilde{M}$ receives a map from~$\mathbf{V}F^{\otimes m}$ induced by the inclusion of~$\mathbf{V}F$ into~$F(0)$.
Post- and precomposing with these maps results in a map
\[
\mu^{z_1,\ldots,z_m}: \mathbf{V}F^{\otimes m} \longrightarrow \overline{\mathbf{V}F}
\]
still depending on~$z\in \C^m\setminus \Delta$, as it should.

We now introduce a little bit more notation to write down the explicit formulas for~$\mu^z$ used to check the axioms of a geometric vertex algebra.
These formulas differ from the above more abstract description in that specific radii of discs appear. 
As in~\cite{CostelloGwilliam} and~\cite{HuangFunctionalAnalyticTheoryI} and other places, we consider a moduli space of a fixed number of discs sitting inside a larger disc.
For~$r_1,\ldots,r_m>0$, let~$D(r_1,\ldots,r_m; R)$ denote the set of~$z\in \C^m$ with
\[
B_{r_1}(z_1)\sqcup \ldots \sqcup B_{r_m}(z_m) \subseteq B_R(0)\displayperiod
\]
For~$r>0$, let
\[ 
P_r : \mathbf{V}F = \bigoplus_{k\in \Z} F(0)_k \longrightarrow F(B_r(0))
\]
be the map sending a finite sum of homogeneous elements to the sum inside~$F(B_r(0))$, that is,~$P_r$ is given by
\begin{ctikzcd}
F(0)_k = \left(\lim_{s>0} F(B_s(0))\right)_k\rar{(\pi_r)_k} & F(B_r(0))_k \subseteq F(B_r(0))
\end{ctikzcd}
on the~$k$-th summand for~$k \in \Z$.
Let
\[
L^R : F(B_R(0)) \longrightarrow \prod_{k\in \Z} F(0)_k = \overline{\mathbf{V}F}
\]
be the map whose~$k$-th component is given by~(\ref{display:FBRZeroToCostalkWeightSpace}).
In terms of the maps~$L^R$,~$P_{r_i}$, the definition of~$\mu^z$ is equivalent to the equation
\begin{align}
&\mu(a_1,z_1,\ldots,a_m,z_m) = \mu^z(a) \notag \\
&= L^R(M^{B_{r_1}(z_1),\ldots,B_{r_m}(z_m)}_{B_R(0)}(\sigma^{B_{r_1}(0)}_{(+z_1)}(P_{r_1}(a_1)),\ldots,\sigma^{B_{r_m}(0)}_{(+z_m)}(P_{r_m}(a_m)))
\end{align}
in~$\Vbar$ for all~$a=a_1\otimes\ldots\otimes a_m \in \mathbf{V}F^{\otimes m}$ and~$r_1,\ldots,r_m,R>0$ with~$z\in D(r_1,\ldots,r_m;R)$. 
This is because~$P_r$ is the lower left composite in the commutative square
\begin{ctikzcd}
  F(0)_k \rar{} \dar{(\pi_r)_k} & F(0) \dar{\pi_r} \\
  F(B_r(0))_k \rar{} & F(B_r(0)) 
\end{ctikzcd}
in which the horizontal maps are the inclusions. 
Note that, for every~$R>0$, an element~$x \in \overline{\mathbf{V}F}$ is uniquely determined by the~$P_R p_k x$ for~$k \in\Z$.
Considering one degree~$k\in \Z$ at a time,
\begin{align} 
&P_R (p_k (\mu^z(a)))\notag\\ 
&= l^R_k(M^{B_{r_1}(z_1),\ldots,B_{r_m}(z_m)}_{B_R(0)}(\sigma^{B_{r_1}(0)}_{(+z_1)}(P_{r_1}(a_1)),\ldots,\sigma^{B_{r_m}(0)}_{(+z_m)}(P_{r_m}(a_m))))
\end{align}
in~$F(B_R(0))_k$ because~$P_R p_k L^R y= l^R_k y$ for all~$y \in F(B_R(0))$.

So far, we know that the maps~$\mu$ are linear maps from~$\mathbf{V}F^{\otimes m}$ to the set of functions on~$\C^m\setminus\Delta$ with values in~$\overline{\mathbf{V}F}$.
Our next goal is holomorphicity. 
\begin{proposition}\label{proposition:MuFromHolomorphicPrefactorizationAlgebraIsHolomorphic}
Let~$a\in \V^{\otimes m}$. Then~$\mu^z(a) = \mu(a,z)$ is a holomorphic function of~$z\in \C^m \setminus \Delta$.
\end{proposition} 
\begin{proof}
A function~$\C^m \setminus \Delta \rightarrow \overline{\mathbf{V}F}$ is holomorphic if and only if its components~$p_k \after f$ are holomorphic for all~$k\in\Z$.
Since~$\mathbf{V}F_k$ is assumed to be discrete,~$p_k \after f$ is holomorphic if it is locally a holomorphic map to a finite-dimensional subspace of~$\mathbf{V}F_k$.
(This notion of holomorphicity coincides with the notion of holomorphicity used in the definition of a geometric vertex algebra.)
 
Let~$z\in \C^m \setminus \Delta$.
Pick~$R>0$ and~$r>0$ s.\,t.~$z\in D(r,\ldots,r; R)$.
Let~$s=r/2$.
It suffices to find a neighborhood~$U$ of~$z$ s.\,t.~$\mu^{z'}(a)$ is a holomorphic function of~$z' \in U$ for some given~$a=a_1\otimes \ldots \otimes a_m$.
We use~$U= \prod^m_{i=1} B_s(z)$ and note that it suffices to prove the holomorphicity of
\[N(z') = M^{B_s(z'_1),\ldots, B_s(z'_m)}_{B_R(0)} (\sigma^{B_s(0)}_{(+z'_1)}(x_1),\ldots,\sigma^{B_s(0)}_{(+z'_m)}(x_m))\]
for~$x_i \in F(B_s(0))$ because we can then plug in~$x_i = P_s a_i$ and postcompose with~$L_R$ to get~$\mu(a,z')$.
By the associativity of~$M$,
\begin{align*}
N(z') = M^{B_r(z_1),\ldots B_r(z_m)}_{B_R(0)} (\rho_1(z'_1)(x_1),\ldots,\rho_m(z'_m)(x_m))\displaycomma
\end{align*}
where
\labelledMapMapsto{\rho_i(z'_i)}{F(B_s(0))}{F(B_r(z_i))}{y}{M^{B_s(z'_i)}_{B_r(z_i)}(\sigma^{B_s(0)}_{(+z'_i)}(y)),}
so~$N$ is holomorphic as the composite of
\begin{itemize} 
\item the map
\MapMapsto{U}{\overline{\bigotimes}^m_{i=1} \mathbf{\BVS}(F(B_s(0)),F(B_r(z_i)))}{(z'_1,\ldots,z'_m)}{\rho_1(z'_1)\otimes \ldots \otimes \rho_m(z'_m)\displaycomma}
which is holomorphic because~$F$ is holomorphically translation-invariant, and the pointwise tensor product of holomorphic functions is holomorphic,
\item the bounded evaluation-at-$a$ map to~$\bigotimes^m_{i=1} F(B_r(z_i))$, and 
\item the bounded map~$M^{B_r(z_1),\ldots B_r(z_m)}_{B_R(0)}$.
\end{itemize}
\end{proof}
\begin{remark}
If~$F$ is holomorphic in the sense of Definition~\ref{definition:HolomorphicallyInvariantPrecosheaves}, then it is holomorphic pointwise, meaning that~$\rho_{U,V}(z)(x)$ is a holomorphic function of~$z \in \mathcal{D}_{U,V}$ for all~$x \in F(U)$. 
The proof of Proposition~\ref{proposition:MuFromHolomorphicPrefactorizationAlgebraIsHolomorphic} works if we only assume that~$F$ is holomorphic pointwise.
\end{remark}
\begin{proposition}\label{proposition:MuFromHolomorphicPrefactorizationAlgebraSatisfiesInsertionAtZero}
$(\mathbf{V}F,\mu)$ satisfies the insertion-at-zero axiom.
\end{proposition}
\begin{proof}
Let~$a\in \mathbf{V}F$. 
For all~$R>0$ and~$k\in \Z$,
\begin{align}\label{equation:FormulaForMuOnDiscAndDegreewise}
  P_R p_k \mu(a,0) &= l^R_k M^{B_r(0)}_{B_R(0)} (P_r a) = l^R_k P_R a = l^R_k P_R \sum_{l\in \Z} p_l a \notag\\
  &= \sum_{l\in \Z} l^R_k P_R p_l a = P_R p_k a\displaycomma
\end{align}
where the last equality holds because~$P_R$ is~$\D^\times$-equivariant, so~$P_R p_l a$ has degree~$l$, and~$l^R_k$ is the identity on the~$k$-th weight space and zero on the others.   
This implies~$\mu(a,0) = a$.
\end{proof}
\begin{proposition} 
$(\mathbf{V}F,\mu)$ is~$\C^\times$-equivariant.
\end{proposition}
\begin{proof}
For~$\D^\times$ instead of~$\C^\times$, this follows since the maps are equivariant w.\,r.\,t.\ this semi-group in an appropriate sense:
For all degrees~$k\in \Z$, it suffices to consider some~$R>0$ with~$z_i \in B_R(z_i)$ for~$i=1,\ldots,m$ with~$r>0$ small enough so that~$z \in D(r,\ldots,r;R)$, and to show equality after application of~$P_R \after p_k$.
\begin{align*}
  &P_Rp_k(q.\mu(a_1,z_1,\ldots,a_m,z_m))\\
  &=q.P_R(p_k(\mu(a_1,z_1,\ldots,a_m,z_m)))\\
  &=q.l^R_k(M^{B_r(z_1),\ldots,B_r(z_m)}_{B_R(0)}(\sigma^{B_r(0)}_{(+z_1)}(P_ra_1),\ldots, \sigma^{B_r(0)}_{(+z_m)}(P_ra_m)))\\
  &=l^R_k(q.M^{B_r(z_1),\ldots,B_r(z_m)}_{B_R(0)}(\sigma^{B_r(0)}_{(+z_1)}(P_ra_1),\ldots, \sigma^{B_r(0)}_{(+z_m)}(P_ra_m)))\\
  \intertext{(equivariance of~$P_R$ and~$p_k$, Equation~\ref{equation:FormulaForMuOnDiscAndDegreewise}, equivariance of~$l^R_k$)}
  &=l^R_kM^{B_{qR}(0)}_{B_R(0)}\sigma^{B_R(0)}_q M^{B_r(z_1),\ldots,B_r(z_m)}_{B_R(0)}(\sigma^{B_r(0)}_{(+z_1)}(P_ra_1),\ldots, \sigma^{B_r(0)}_{(+z_m)}(P_ra_m))\\
  \intertext{(definition of the action, see Definition~\ref{definition:DtimesActionOnValueOnDiscOfInvariantPrecosheaf})}
  &=l^R_kM^{B_{qR}(0)}_{B_R(0)}M^{B_{qr}(qz_1),\ldots,B_{qr}(qz_m)}_{B_{qR}(0)}(\sigma^{B_r(z_1)}_{(q\cdot)}(\sigma^{B_r(0)}_{(+z_1)}(P_ra_1)),\ldots,\\
  &\quad\quad\quad\quad\sigma^{B_r(z_m)}_{(q\cdot)}( \sigma^{B_r(0)}_{(+z_m)}(P_ra_m) )) \\
  \intertext{(invariance of~$F$, see Remark~\ref{remark:InvariantPrefactorizationAlgebrasAndMAryOperations})} 
  &=l^R_kM^{B_{qr}(qz_1),\ldots,B_{qr}(qz_m)}_{B_{R}(0)}(\sigma^{B_{qr}(0)}_{(+qz_1)}(\sigma^{B_r(0)}_{(q\cdot)}(P_ra_1)),\ldots,\sigma^{B_{qr}(0)}_{(+qz_m)}( \sigma^{B_r(0)}_{(q\cdot)}(P_ra_m) ))\\
\intertext{(associativity of~$F$, composition law in~$\C^\times \ltimes \C$)}
  &=l^R_k M^{B_r(qz_1),\ldots,B_r(qz_m)}_{B_R(0)}(M^{B_{qr}(qz_1)}_{B_r(qz_1)}(\sigma^{B_{qr}(0)}_{(+qz_1)}(\sigma^{B_r(0)}_{(q\cdot)}(P_ra_1))),\ldots,\\
  &\quad\quad\quad\quad M^{B_{qr}(qz_m)}_{B_r(qz_m)}(\sigma^{B_{qr}(0)}_{(+qz_m)}( \sigma^{B_r(0)}_{(q\cdot)}(P_ra_m) )))\quad\quad   \text{(associativity)}  \\
  &=l^R_k M^{B_r(qz_1),\ldots,B_r(qz_m)}_{B_R(0)}(\sigma^{B_{r}(0)}_{(+qz_1)}(M^{B_{qr}(0)}_{B_r(0)}(\sigma^{B_r(0)}_{(q\cdot)}(P_ra_1))),\ldots,\\
  &\quad\quad\quad\quad \sigma^{B_{r}(0)}_{(+qz_m)}(M^{B_{qr}(0)}_{B_r(0)}(\sigma^{B_r(0)}_{(q\cdot)}(P_ra_m) ))) \quad\quad\quad\text{(invariance)}\\
  &=l^R_k M^{B_r(qz_1),\ldots,B_r(qz_m)}_{B_R(0)}(\sigma^{B_{r}(0)}_{(+qz_1)}(q.P_ra_1),\ldots,\sigma^{B_{r}(0)}_{(+qz_m)}(q.P_ra_m) )\\
  &=l^R_k M^{B_r(qz_1),\ldots,B_r(qz_m)}_{B_R(0)}(\sigma^{B_{r}(0)}_{(+qz_1)}(P_r(q.a_1)),\ldots,\sigma^{B_{r}(0)}_{(+qz_m)}(P_r(q.a_m)) )\\
  &=P_Rp_k\mu(q.a_1,qz_1,\ldots,q.a_m,qz_m) \quad\quad\text{(Equation \ref{equation:FormulaForMuOnDiscAndDegreewise})}
\end{align*}  
Equivariance for~$\C^\times$ follows from the uniqueness of analytic continuation.
\end{proof}
\begin{proposition}\label{proposition:MuFromHolomorphicPrefactorizationAlgebraIsAssociative}
$\mu$ satisfies the associativity axiom.
\end{proposition}
The map~$\Phi$ in the following proof plays a very similar role as the map denoted by the same letter in Proposition~5.3.6 of~\cite{CostelloGwilliam}.
\begin{proof}
Let~$a_1, \ldots, a_m \in \mathbf{V}F$,~$b_1,\ldots, b_n \in \mathbf{V}F$.
Let
\[
\mathcal{A}_{m,n} = \{ (z,w) \in (\C^m \setminus \Delta) \times (\C^n\setminus \Delta) \mid \max_{1\leq j \leq n} |w_j| < \min_{1\leq i \leq m} |z_i - z_{m+1}| \}
\]
denote the set of tuples of points in~$\C$ as in the associativity axiom of a geometric vertex algebra.
The summands in the associativity axiom, viewed as functions of~$(z,w) \in \mathcal{A}_{m,n}$, are elements of
\(
\mathcal{O}(\mathcal{A}_{m,n}; \overline{\mathbf{V}F})
\)
and we prove absolute convergence in a completely normable subobject. 
By the definition of the bornology on~$\mathcal{O}(\mathcal{A}_{m,n}; \mathbf{V}F)$, this implies that, for every compact subset~$K \subseteq \mathcal{A}_{m,n}$ and~$k \in \Z$, there is a completely normable subobject~$V_{K,k}$ of~$\mathbf{V}F_k$ in which absolute convergence takes place.
The discreteness of~$\mathbf{V}F_k$ implies that the bounded unit disc of~$V_{K,k}$ is contained in a finite-dimensional subspace of~$\mathbf{V}F_k$ so~$V_{K,k}$ is finite-dimensional.  

It suffices to prove convergence of functions on a neighborhood~$U$ of~$(z,w)$ for every~$(z,w) \in \mathcal{A}_{m,n}$. 
Let~$S > T_2 >T_1>0$ all be smaller than~$\min_{1\leq i \leq m} |z_i - z_{m+1}|$, but only slightly smaller, so that
\[
w_1,\ldots,w_n \in B_{T_1}(0) \quad (\subset B_{T_2}(0) \subset B_S(0))\displayperiod
\]
Let~$r>0$ be small enough so that~$w \in D(r,\ldots, r; T_1)$ and the~$B_r(z_i), i=1,\ldots, m,$ are pairwise disjoint and disjoint from~$B_S(z_{m+1})$.
Let~$R>0$ be big enough so that
\[
z \in D(r,\ldots,r,S;R)\displayperiod 
\]
Let~$s=r/2$ and
\[
U = \prod^m_{i=1} B_s(z_i) \times B_{S-T_2}(z_{m+1}) \times \prod^n_{j=1} B_s(w_j)
\]
which is a subset of~$\mathcal{A}_{m,n}$ as we now check.
Let~$(z',w') \in U$. 
For all~$i,j$ with~$1 \leq i \leq m$ and~$1\leq j \leq n$,
\begin{align*}
  |w'_j| &< |w_j| + s = |w_j| + r -s < T_1 - s  \\
  &< S = S+r - 2s \leq |z_i - z_{m+1}| -2s < |z'_i - z'_{m+1}| 
\end{align*}
so~$(z',w') \in \mathcal{A}_{m,n}$.

We show convergence of holomorphic functions on~$U$ with values in~$\overline{\mathbf{V}F}$ by constructing a bounded linear map
\[
\Phi : \BVS(F(B_{T_1}(0)),F(B_{T_2}(0))) \longrightarrow \mathcal{O}(U;\overline{\mathbf{V}F})
\]
with~$\Phi(i^{T_1}_{T_2})$ given by the r.\,h.\,s. of the associativity axiom,
 see~(2) 
 of~\cite{GeometricVertexAlgebras}, 
 and~$\Phi(i^{T_1}_{T_2}\after l^{T_1}_k)$ given by the~$k$-th summand of the associativity axiom for all~$k\in Z$,
  see~(1) 
  of~\cite{GeometricVertexAlgebras}.
The existence of such a map implies the associativity axiom on~$U$ because bounded linear maps map absolutely convergent sums to absolutely convergent sums and~$i^{T_1}_{T_2} = \sum_{k\in \Z} i^{T_1}_{T_2} \after l^{T_1}_k$ by Proposition~\ref{proposition:SumOfWeightSpaceProjections}.
Let 
\begin{align*}
  &(\Phi(f))(z',w')\\
  &= L^R M^{B_s(z'_1),\ldots,B_s(z'_{m}),B_{T_2}(z'_{m+1})}_{B_R(0)} (\sigma^{B_s(0)}_{(+z'_1)}(P_sa_1),\ldots,\sigma^{B_s(0)}_{(+z'_m)}(P_sa_m),\\
  &\phantom{=L^RM()} \sigma^{B_{T_2}(0)}_{(+z'_{m+1})}(f(M^{B_s(w'_1),\ldots,B_s(w'_n)}_{B_{T_1}(0)}(\sigma^{B_s(0)}_{(+w'_1)}(P_sb_1),\ldots,\sigma^{B_s(0)}_{(+w'_n)}(P_sb_n))))\displayperiod
\end{align*} 
For~$f=i^{T_1}_{T_2}$, we use associativity, translation invariance, and associativity again:
\begin{align*}
&(\Phi(f))(z',w')\\
&=L^R M^{B_s(z'_1),\ldots,B_s(z'_{m}),B_{T_2}(z'_{m+1})}_{B_R(0)} (\sigma^{B_s(0)}_{(+z'_1)}(P_sa_1),\ldots,\sigma^{B_s(0)}_{(+z'_m)}(P_sa_m),\\
&\phantom{=L^R M()} \sigma^{B_{T_2}(0)}_{(+z'_{m+1})}(M^{B_s(w'_1),\ldots,B_s(w'_n)}_{B_{T_2}(0)}(\sigma^{B_s(0)}_{(+w'_1)}(P_sb_1),\ldots,\sigma^{B_s(0)}_{(+w'_n)}(P_sb_n)))\\
&= L^R M^{B_s(z'_1),\ldots,B_s(z'_{m}),B_{T_2}(z'_{m+1})}_{B_R(0)} (\sigma^{B_s(0)}_{(+z'_1)}(P_sa_1),\ldots,\sigma^{B_s(0)}_{(+z'_m)}(P_sa_m),\\
&\phantom{=} M^{B_s(w'_1+z'_{m+1}),\ldots,B_s(w'_n+z'_{m+1})}_{B_{T_2}(z'_{m+1})}(\sigma^{B_s(0)}_{(+w'_1+z'_{m+1})}(P_sb_1),\ldots,\sigma^{B_s(0)}_{(+w'_n+z'_{m+1})}(P_sb_n)))\\
&= L^R M^{B_s(z'_1),\ldots,B_s(z'_{m}),B_{T_2}(z'_{m+1}),B_s(w'_1+z'_{m+1}),\ldots,B_s(w'_n+z'_{m+1})}_{B_R(0)} (\\
&\phantom{=} \sigma^{B_s(0)}_{(+z'_1)}(P_sa_1),\ldots,\sigma^{B_s(0)}_{(+z'_m)}(P_sa_m),\sigma^{B_s(0)}_{(+w'_1+z'_{m+1})}(P_sb_1),\ldots,\sigma^{B_s(0)}_{(+w'_n+z'_{m+1})}(P_sb_n))\\
&= \mu(a_1,z'_1,\ldots,a_m,z'_m, b_1,w'_1+z'_{m+1},\ldots, b_n, w'_n+z'_{m+1})
\end{align*}
For~$k\in \Z$ and~$f=i^{T_1}_{T_2} \after l^{T_1}_k$, we consider the argument of~$\sigma^{B_{T_2}(0)}_{(+z'_{m+1})}$ which is
\begin{align*}
&i^{T_1}_{T_2} \after  l^{T_1}_k(M^{B_s(w'_1),\ldots,B_s(w'_n)}_{B_{T_1}(0)}(\sigma^{B_s(0)}_{(+w'_1)}(P_sb_1),\ldots,\sigma^{B_s(0)}_{(+w'_n)}(P_sb_n)))\\
\intertext{and, because of $i^{T_1}_{T_2} \after  l^{T_1}_k =  l^{T_2}_k \after i^{T_1}_{T_2}$ and associativity, equal to}
&P_{T_2}(p_k(\mu(b_1,w'_1,\ldots,b_n,w'_n)))\displayperiod
\end{align*}
Plugging in, this implies that
\begin{align*}
&(\Phi(f))(z',w')\\
&=L^R M^{B_s(z'_1),\ldots,B_s(z'_{m}),B_{T_2}(z'_{m+1})}_{B_R(0)} (\sigma^{B_s(0)}_{(+z'_1)}(P_sa_1),\ldots,\sigma^{B_s(0)}_{(+z'_m)}(P_sa_m),\\
&\phantom{=L^R M()} \sigma^{B_{T_2}(0)}_{(+z'_{m+1})}(P_{T_2}(p_k(\mu(b_1,w'_1,\ldots,b_n,w'_n)))))\\
&=\mu(a_1,z_1,\ldots,a_m,z_m,p_k(\mu(b_1,w'_1,\ldots,b_n,w'_n)),z_{m+1})
\end{align*}
which is the~$k$-th summand in the associativity axiom. 

Now, we check that~$\Phi$ is a bounded linear map and actually maps to the space of holomorphic functions~$\mathcal{O}(U;\overline{\mathbf{V}F})$ by writing~$\Phi$ as a composite of bounded linear maps. 
This argument is similar to the proof of Proposition~\ref{proposition:MuFromHolomorphicPrefactorizationAlgebraIsHolomorphic} about holomorphicity.
By associativity,
\begin{align*}
  &(\Phi(f))(z',w')\\
  &= L^R M^{B_r(z_1),\ldots,B_r(z_{m}),B_S(z_{m+1})}_{B_R(0)} (\rho_1(z'_1)(P_sa_1),\ldots,\rho_m(z'_m)(P_sa_m)),\\
  &\phantom{L^R M()} \rho_{m+1}(z'_{m+1})(f(M^{B_r(w_1),\ldots,B_r(w_n)}_{B_{T_1}(0)}(\tau_1(w'_1)(P_sb_1),\ldots,\tau_n(w'_n)(P_sb_n))))\displaycomma
\end{align*}
where the map
\labelledMapMapsto{\rho_i}{B_s(z_i)}{\BVS(F(B_s(0)),F(B_r(z_i)))}{z'_i}{M^{B_s(z'_i)}_{B_r(z_i)}\after \sigma^{B_s(0)}_{(+z'_i)}}
is holomorphic for~$i=1,\ldots, m$, and so are
\labelledMapMapsto{\rho_{m+1}}{B_{S-T_2}(z_{m+1})}{\BVS(F(B_{T_2}(0)),F(B_{S}(z_i)))}{z'_{m+1}}{M^{B_{T_2}(z'_{m+1})}_{B_S(z_{m+1})}\after\sigma^{B_{T_2}(0)}_{(+z'_{m+1})}}
and
\labelledMapMapsto{\tau_j}{B_s(w_j)}{\BVS(F(B_s(0)),F(B_r(w_j)))}{w'_j}{M^{B_s(w'_j)}_{B_r(w_j)}\after\sigma^{B_s(0)}_{(+w'_j)}}
for~$j=1,\ldots, n$.
Thus the map
\MapMapsto{\prod^n_{j=1} B_s(w_j)}{\overline{\bigotimes}^n_{j=1} F(B_r(w_j))}{(w'_1,\ldots, w'_n)}{\tau_1(w'_1)(P_sb_1) \otimes \ldots \otimes \tau_n(w'_n)(P_sb_n)}
is holomorphic.
Postcomposing with~$M^{B_r(w_1),\ldots,B_r(w_n)}_{B_{T_1}(0)}$ 
and~$f$ defines a bounded linear map
\[
 \BVS(F(B_{T_1}(0)),F(B_{T_2}(0)) \longrightarrow \mathcal{O}\left(\prod^n_{j=1} B_s(w_j); F(B_{T_2}(0))\right)\displayperiod
\]
Taking the external product with
\[
  \rho_{m+1} \in \mathcal{O}(B_{S-T_2}(z_{m+1}); \BVS(F(B_{T_2}(0)),F(B_{S}(z_i))))
\]
is a bounded linear map with target 
\[
  \mathcal{O}\left(B_{S-T_2}(z_{m+1})\times \prod^n_{j=1} B_s(w_j)\ ;\ \BVS\left(F(B_{T_2}(0)),F(B_{S}(z_i))\right) \barotimes F(B_{T_2}(0))\right)
\] 
from which we can postcompose with the evaluation pairing to get an element of
\[
  \mathcal{O}(B_{S-T_2}(z_{m+1})\times \prod^n_{j=1} B_s(w_j); F(B_{S}(z_i)))\displayperiod
\]
Again using the holomorphicity of~$F$, the map
\MapMapsto{\prod^m_{i=1} B_s(z_i)}{\overline{\bigotimes}^m_{i=1} F(B_r(z_i))}{(z'_1,\ldots, z'_m)}{\rho_1(z'_1)(P_sb_1) \otimes \ldots \otimes \rho_m(z'_m)(P_sb_m)}
is holomorphic, and taking the external product with it is  a bounded linear map with target 
\[
  \mathcal{O}(U;\overline{\otimes}^m_{i=1} F(B_r(z_i)) \otimes F(B_{S}(z_i)))\displayperiod
\]
Postcomposing with~$L^R \after M^{B_r(z_1),\ldots,B_r(z_{m}),B_S(z_{m+1})}_{B_R(0)}$ gives
\(
  \Phi(f) \in \mathcal{O}(U; \overline{\mathbf{V}F}).
\)
\end{proof}
\begin{definition}
Let~$F$ be a holomorphic affine-linearly invariant factorization algebra on~$\C$.
We say that~$F$ has discrete weight spaces if~$\mathbf{V}F_k$ is discrete for all~$k \in \Z$. 
If~$F$ has discrete weight spaces, we say that~$F$ has \emph{meromorphic operator product expansion (OPE)} if~$(\mathbf{V}F, \mu)$ satisfies the meromorphicity axiom of a geometric vertex algebra (see Definition~1.3 of \cite{GeometricVertexAlgebras}), where~$\mu$ is the sequence of multiplication maps for~$\mathbf{V}F$ constructed above.
\end{definition} 
We sometimes call holomorphic prefactorization algebras with meromorphic OPE \emph{meromorphic prefactorization algebras}, with the caveat that, in contrast to the language for functions, a meromorphic prefactorization algebra is a special kind of holomorphic prefactorization algebra. 
\begin{proposition}
If~$F$ is a holomorphic prefactorization algebra on~$\C$ with discrete weight spaces and meromorphic OPE, then~$(\mathbf{V}F, \mu)$ is a geometric vertex algebra.
\end{proposition}
\begin{proof}
Permutation invariance holds because the multiplication maps in a prefactorization algebra are compatible with the braiding.
We have already checked all the other axioms except meromorphicity which we now assume.
\end{proof}
Recall from 
Proposition~2.3 
 of~\cite{GeometricVertexAlgebras}
 that~$\mathbf{V}F$ is meromorphic if it is bounded from below.
Being bounded from below is the assumption made in~\cite{CostelloGwilliam} when constructing a vertex algebra from a prefactorization algebra.
While less general, this assumption has the pleasing feature that it does not make reference to the multiplication maps of the prefactorization algebra. 
\Subsection{From Geometric Vertex Algebras to Prefactorization Algebras}\label{subsection:FromGeometricVertexAlgebrasToPrefactorizationAlgebras}
Let~$\V$ be a geometric vertex algebra.
We construct a holomorphic prefactorization algebra~$\mathbf{F}\V$ on~$\C$. 
In Section~\ref{subsection:BackToGeometricVertexAlgebras}, we prove that~$\mathbf{F}\V$ has meromorphic OPE and discrete weight spaces by relating it to~$\V$. 
We prove that~$\mathbf{F}\V$ is a factorization algebra in Section~\ref{section:FactorizationAlgebras}.

The precosheaf~$\mathbf{F}\V$ is a quotient of the precosheaf~$\mathbf{E}\V$ of \emph{expressions} with labels in~$\V$, which we define first. It only makes use of the underlying~$\Z$-graded vector space of~$\V$ and not of its multiplication maps.
\begin{definition}\label{definition:E}
Let~$\V$ be a~$\Z$-graded vector space.
 We define a precosheaf on~$\C$ with values in the category of complete bornological spaces. 
Let~$U\subseteq \C$ be open.
The precosheaf~$\mathbf{E}\V$ of \emph{expressions} is defined by
\[
\mathbf{E}\V(U) = \bigoplus_{m\geq 0} (\mathcal{O}'(U^m \setminus \Delta) \otimes \V^{\otimes m})_{\Sigma_m}\displaycomma
\]
which is complete by the following proposition. 
Here, the action of~$\sigma \in \Sigma_m$ on~$\mathcal{O}'(U^m \setminus\Delta) \otimes \V^{\barotimes m}$ is given by
\[
\sigma . (\alpha \otimes a_1 \otimes \ldots \otimes a_m) =\sigma_*\alpha \otimes a_{\sigma(1)} \otimes \ldots \otimes a_{\sigma(n)}\displayperiod
\] 
The extension maps of~$\mathbf{E}\V$ are the pushforward of analytic functionals, dual to restriction of functions, and the identity on the~$\V^{\otimes m}$.
\end{definition}
If we have a fixed~$\V$ in mind, we also write~$E$ for~$\mathbf{E}\V$.
\begin{proposition}\label{proposition:EofUisComplete}
  ~$E(U)$ is complete for all open~$U \subseteq \C$. 
\end{proposition}
\begin{proof}
Analytic functionals are complete as the dual of a locally convex topological vector space. 
The tensor products are again complete because taking the bornological tensor product with a discrete space preserves completeness and~$\V$ is defined to be discrete. 
  The quotient of a complete space by the action of a finite group is complete, too, so~$(\mathcal{O}'(U^n \setminus \Delta) \otimes \V^{\otimes n})_{\Sigma_n}$ is complete.
  Infinite direct sums of complete spaces are again complete, 
  finishing our argument that~$E(U)$ is complete.
\end{proof}
We now turn~$E$ into a prefactorization algebra.
To multiply~$\alpha \in \mathcal{O}'(U^m \setminus \Delta)$ and~$\beta \in \mathcal{O}'(V^n \setminus \Delta)$ for~$U\sqcup V \subseteq W$, we use the external product of functionals, see Equation~\ref{equation:ExternalProductOfAnalyticFunctionals}, to get~$\alpha \times \beta \in \mathcal{O}'((U^m \setminus \Delta) \times (V^n \setminus \Delta))$ which we can then pushforward to an element~$\alpha\beta \in \mathcal{O}'((U\sqcup V)^n \setminus \Delta)$ along the embedding of~$(U^m \setminus \Delta) \times (V^n \setminus \Delta)$ into~$(U\sqcup V)^{m+n} \setminus \Delta$
induced by the associativity isomorphism~$\C^m \times \C^n \cong \C^{m+n}$.  
This embedding is well-defined because~$U$ and~$V$ are disjoint. 
\begin{proposition}
If~$\V$ is a~$\Z$-graded vector space, then~$\mathbf{E}\V$ is a prefactorization algebra on~$\C$ with values in the symmetric monoidal category of complete bornological spaces with the underlying precosheaf of Definition~\ref{definition:E} whose multiplication maps are
\labelledMapMapsto{\mu^{U,V}_W}
{E\V(U) \otimes E\V(V)}{E\V(W)}{ [\alpha \otimes a] \otimes [\beta \otimes b]}{[\alpha\beta \otimes a \otimes b ]}
where~$\alpha \otimes a \in \mathcal{O}'(U^m \setminus\Delta) \otimes \V^{\barotimes m}$ and~$\beta \otimes b \in \mathcal{O}'(V^n \setminus\Delta) \otimes \V^{\barotimes n})$.
Its unit is~$1$ in the summand corresponding to~$m=0$ which can be identified with~$\C$ in a natural way and which is present for~$U=\emptyset$, too. 
\end{proposition}
\begin{proof}
The external product of functionals is compatible with pushforwards along complex-analytic embeddings. 
This implies that the multiplication maps for~$E$ are well-defined w.\,r.\,t.\ the permutation actions.
Taking external product of functionals is a bounded linear map and the braiding of~$\CBVS$ is as well, so the multiplication maps of~$E$ are bounded linear maps.
Unitality and associativity follow from the behavior of the external product of functionals.

Compatibility with the braiding is enforced by modding out the action of the permutation groups, let~$\tau\in \Sigma_{m+n}$ be the permutation moving the first~$m$ elements past the last~$n$ elements: 
We want to show that
\begin{align*}
\beta\alpha(f) &= \beta(y \mapsto \alpha(x \mapsto f(y,x)))\\
\intertext{is equal to}
\tau_*(\alpha\beta)(f)&= \alpha\beta(\tau^*f)= \alpha(x \mapsto \beta(y \mapsto \tau^*f(x,y)))= \alpha(x \mapsto \beta(y \mapsto f(y,x)))
\end{align*}
for
\begin{align*}
&f \in \mathcal{O}((U\sqcup V)^{m+n} \setminus\Delta)\\
&\alpha \in \mathcal{O}'(U^m \setminus\Delta),\quad \beta \in \mathcal{O}'(V^n \setminus\Delta)\displayperiod
\end{align*}
The functionals~$\beta\alpha$ and~$\tau_*(\alpha\beta)$ are the pushforwards of
\begin{align*}
  &\beta\times\alpha, \tau_*(\alpha\times \beta) \in \mathcal{O}'((V^n \setminus \Delta) \times (U^m \setminus \Delta))\displayperiod
\end{align*}
Thus, it suffices to check the desired equality on the dense subset of holomorphic functions~$f$ which are a finite sum of functions~$h(y)g(x)$ of~$x \in U^m \setminus \Delta$ and~$y\in V^n \setminus \Delta$ for~$g$ and~$h$ holomorphic.
For~$f(y,x) = h(y)g(x)$, 
\[
  \beta\times \alpha (f) = \alpha(g)\beta(h) = (\tau_*(\alpha \times\beta))(f) \displayperiod
\]
To identify the underlying precosheaf of~$E$ with the precosheaf defined earlier, we note that multiplying with the unit in~$\C \cong \mathcal{O}'(\mathrm{pt}) = \mathcal{O}'(\emptyset^0 \setminus \Delta_0)$ is the pushforward of analytic functionals.
\end{proof}
Now using the multiplication maps of~$\V$, we define the \emph{evaluation map} 
\[
\ev_U: \mathbf{E}\V(U) \rightarrow \Vbarbb  
\]
for~$U \subseteq \C$ open.
Let~$\alpha \otimes a \in \mathcal{O}'(U^m \setminus \Delta) \otimes \V^m$. 
As part of the structure of a geometric vertex algebra, we have a holomorphic function~$\mu(a): U^m \setminus \Delta \rightarrow \Vbar$.
We would like to evaluate~$\alpha$ on it to obtain an element of~$\Vbar$.
Let~$k\in \Z$. The holomorphic function~$p_k \after \mu(a)$ takes values in a finite-dimensional subspace~$F \subseteq \V_k$. 
The evaluation pairing tensored with~$F$ gives a map ($X := U^m \setminus \Delta$)
\[
\mathcal{O}'(X) \otimes \mathcal{O}(X;F) \cong \mathcal{O}'(X) \otimes \mathcal{O}(X) \otimes F \rightarrow \C \otimes F \cong F 
\]
so we define the evaluation of~$\alpha$ on~$p_k \after \mu(a)$ to be the image of~$\alpha \otimes (p_k \after \mu(a))$ under this map to get an element of~$\V_k$ for all~$k \in \Z$. 
Let~$\alpha(\mu(a))$ denote the resulting element of~$\Vbar$.
What we have just described more generally applies to any open~$X \subseteq \C^m$ and any holomorphic function~$f \in \mathcal{O}(X;\Vbar)$ and~$\alpha \in \mathcal{O}'(X)$ and defines~$\alpha(f) \in \Vbar$, where~$\V$ is a~$\Z$-graded vector space. 
It is clear that~$\alpha(f)$ depends linearly and boundedly on~$\alpha \otimes f$, by the universal property of the product~$\Vbar$ this means linearity and boundedness in each component. 
It follows that~$\alpha(\mu(a))$ depends linearly and boundedly on~$\alpha \otimes a$, recall that~$\V^{\otimes m}$ is given the discrete bornology. 
If~$(e_i)_i$ is a basis of a finite-dimensional subspace~$F$ in which~$p_k \after f$ takes values for some~$k\in \Z$, then
\[
  \alpha(f)_k = \sum_i \alpha(e^*_i(p_k \after f))e_i\displayperiod
\]
The permutation invariance of~$\mu$ and the universal property of the direct sum give a well-defined and bounded linear map 
\labelledMapMapsto{\ev_U}{\mathbf{E}\V(U)}{\Vbar}{ [\alpha\otimes a] }{\alpha(\mu(a))\displayperiod}
The collection of the~$\ev_U$ for~$U\subseteq \C$ open defines a map of precosheaves from~$E$ to the constant precosheaf~$\const \Vbar$.
\begin{definition}
  The precosheaf~$R$ on~$\C$ is defined as the kernel of~$\ev: E(U)\rightarrow \const \overline{\mathcal{V}}$, that is,
  \[
    R(U) = \ker \ev_U \subseteq E(U)
  \]
  for~$U\subseteq \C$ open.
  We call the elements of~$R(U)$~\emph{relations} on~$U$. 
\end{definition}
\begin{remark}\label{remark:ClosednessOfRelations}
  Note that~$R(U) \subseteq E(U)$ is b-closed as the kernel of a bounded linear map since such maps are continuous w.\,r.\,t.\ the b-topology. 
\end{remark}
The multiplication on~$E$ does not induce a multiplication on~$E/R$ in general.
\begin{example}\label{example:NoMultiplicationUsingGlobalRelations} 
  We consider the case that there are~$a,b\in \mathcal{V}$ s.\,t.~$a_{(m)}b \neq 0$ for some~$m\geq 0$.
  For example, this is the case for~$\V$ the free boson (also known as the Heisenbeg vertex algebra),~$a=b$ the generator and~$m=1$.
  Let
  \begin{align*}
  U_1 &= B_2(0) \setminus \overline{B_1(0)}\\
  U_2 &= B_1(0)\\
  W &= B_2(0) 
  \end{align*}
  so that we have~$U_1\cap U_2 = \emptyset$ and~$U_1,U_2 \subseteq W$. 
  Our goal is to show that there is no dotted arrow making
  \begin{ctikzcd}
  E(U_1) \otimes E(U_2)                 \dar{} \rar{} & E(W) \dar{} \\
  E(U_1)/R(U_1) \otimes E(U_2)/R(U_2)        \arrow[r,dashrightarrow] & E(W)/R(W)
  \end{ctikzcd}
  commute, where the left map is the tensor product of the quotient maps and the right map is the quotient map.
  
  Let~$\alpha \in \mathcal{O}'(U_1)$ and~$\beta \in \mathcal{O}'(U_2)$ be the analytic functionals
  \[
  \alpha : h \in \mathcal{O}(U_1) \mapsto \int_{\frac{3}{2} S^1} z^m h(z)dz 
  \]
  resp.
  \[
  \beta: h \in \mathcal{O}(U_2) \mapsto h(0)\displayperiod
  \]
  Then~$\alpha\otimes a$ and~$\beta\otimes b$ represent classes~$x \in E(U_1)$ resp.~$y\in E(U_2)$.
  The image of~$x\otimes y$ in~$E(W)$ under~$M^{U_1,U_2}_W$ is represented by $\alpha\beta \otimes a \otimes b$, where
  \[
    \alpha\beta : h \in \mathcal{O}(W^2 \setminus \Delta) \mapsto \alpha[z \mapsto \beta[w \mapsto h(z,w)]] = \int_{\frac{3}{2} S^1} z^m h(z,0)dz\displayperiod
  \]
  The image of this class in~$E(W)/R(W)$ is non-zero because this space maps injectively to~$\Vbar$ by construction and the image there is
  \[
  \ev_{W}(M^{U_1,U_2}_W([x\otimes y])) = \int_{\frac{3}{2}S^1} z^m\mu(a,z,b,0) dz = a_{(m)} b \neq 0\displayperiod
  \] 
  If there were a dotted arrow making the diagram commute, then the lower composition would give zero since the image of~$x$ in~$E(U_1)/R(U_1)$ is zero because
  \[
  ev_{U_1}(x) = \int_{\frac{3}{2}S^1} z^m\mu(a,z) dz = 0
  \]
  since the integrand is defined on~$\C$ and holomorphic.
\end{example}
This example suggests that we only impose those relations among elements of~$E(U)$ which are visible on simply-connected subsets. 
We choose to focus on round discs.  
\begin{definition}
A \emph{(round, open) disc} is a subset~$D \subseteq \C$ which is equal to
\[
B_r(z) = \{ w \in \C \mid |z - w| < r \}
\] 
for some~$z \in \C$ and some real number~$r > 0$.
\end{definition}
\begin{definition}\label{definition:F}
The multiplication maps of~$E$ assemble to give a map
\[
\bigoplus_{d,W} R(d) \otimes E(W) \longrightarrow E(U)
\]
where the direct sum runs over all discs~$d\subseteq U$ and open subsets~$W\subseteq U$ s.\,t.~$d\cap W = \emptyset$.
Let~$R^{disc}(U) \subseteq E(U)$ be the image of this map,
\[
R^{disc}(U) = \im \left( \bigoplus_{d,W} R(d)\otimes E(W)  \longrightarrow E(U) \right)\displayperiod
\]
We call~$R^{disc}$ the precosheaf of \emph{relations on discs}.  
Let~$F$ be the precosheaf of complete bornological vector spaces defined by
\[
  F(U)=E(U)/\overline{R^{disc}(U)}
\]
for~$U \subseteq \C$ open. 
If we want to make the dependence on~$\V$ explicit, we use the notation~$\mathbf{F}\V$ for~$F$.
\end{definition}
Note that~$F(U)$ is complete because~$E(U)$ is complete by Proposition~\ref{proposition:EofUisComplete} and because we are modding out by the b-closure of~$R^{disc}(U)$. 
The author does not know if~$R^{disc}(U)$ is b-closed or not. If~$U$ is a disc, then~$R^{disc}(U) = R(U)$ by Lemma~\ref{lemma:RelationsOnDisc} further below, and~$R(U)$ is b-closed as the kernel of a bounded map.
To prove that~$F$ inherits multiplication maps from~$E$, we slightly rephrase its definition so that it only refers to complete bornological spaces.
Let~$U\subseteq \C$ be open. 
The source of the map 
\[
\bigoplus_{d,W} R(d) \otimes E(W) \longrightarrow E(U)
\]
as in the definition of~$R^{disc}(U)$ is not complete in general (e.\,g.\,~$\V = \C$,~$W$ non-empty),  
but
\[
  S^{disc}(U) := \bigoplus_{d,W} R(d) \barotimes E(W)\displaycomma
\]
with~$d,W$ as above, is a completion. 
The space of relations~$R(U) \subseteq E(U)$ is b-closed, because it is the intersection of the kernels of the components of~$\ev_U$, each of which is a bounded map, see its construction.
Since~$E(d)$ is complete, its b-closed subspace~$R(d)$ is complete, too.
\begin{proposition}\label{proposition:FAsCokernelOfEByS}
  The quotient map~$E(U) \rightarrow F(U)$ is a cokernel of the map
  \[
    S^{disc}(U) \rightarrow E(U)
  \] 
  given by the sum of the multiplication maps.
\end{proposition}
\begin{proof}
  This amounts to~$\overline{R^{disc}(U)} = \overline{I}$ where~$I$ is the image of~$S^{disc}(U)$ in~$E(U)$.
  Let~$d,W$ be as above. 
  The algebraic tensor product~$R(d) \otimes E(W)$ maps to its completion~$R(d) \barotimes E(W)$, so~$R^{disc}(U) \subseteq I$, hence~$\overline{R^{disc}(U)} \subseteq \overline{I}$.
  Every algebraic tensor product has dense image in its completion, implying that~$\overline{R^{disc}(U)} \supseteq I$ and~$\overline{R^{disc}(U)}\supseteq \overline{I}$.
\end{proof}
\begin{proposition}\label{proposition:FIsAPrefactorizationAlgebra}
  $F$ is a prefactorization algebra on~$\C$ with its multiplication maps induced from those of~$E$.
  The multiplication maps of~$E$ induce unique multiplication maps on~$F$, also denoted~$M$, such that the square
  \begin{ctikzcd}
     E(U) \barotimes E(V) \rar{M^{U,V}_X} \dar{} & E(X) \dar{} \\
     F(U) \barotimes F(V) \rar{M^{U,V}_X}  & F(X) 
  \end{ctikzcd}
  commutes for all~$X \subseteq \C$ open and~$U,V\subseteq X$ open and disjoint, where the vertical maps are the completed tensor product of the quotient maps resp.\ the quotient map.
\end{proposition}
\begin{proof}
  Let~$U$ and~$V$ be open and disjoint and contained in an open~$X \subseteq \C$.
  Since cokernels in~$\CBVS$ commute with~$\barotimes$ in both variables separately, Proposition~\ref{proposition:FAsCokernelOfEByS} implies that the map~$E(U) \barotimes E(V) \rightarrow F(U) \barotimes F(V)$ is the cokernel of the multiplication map
  \[
    T := S^{disc}(U) \barotimes E(V) \oplus E(U) \barotimes S^{disc}(V) \rightarrow E(U) \barotimes E(V)\displayperiod
  \]
  For~$Y \subseteq \C$ open, let~$q_Y: E(Y) \rightarrow F(Y)$ denote the quotient map.
  We want to see that~$T$ maps to zero under~$q_X \after M^{U,V}_X$, as it then follows that~$F$ is a prefactorization algebra because~$E$ is a prefactorization algebra. 
  To show that~$S^{disc}(U) \barotimes E(V)$ maps to zero, let~$d,W \subseteq U$ be as in the definition of~$S^{disc}(U)$.
  Using the associativity of~$E$, we see that the map on the summand~$(R(d) \barotimes E(W)) \barotimes E(V)$ factors through the summand~$R(d) \barotimes E(W \sqcup V)$ of~$S^{disc}(X)$.
  To show that~$E(U) \barotimes S^{disc}(V)$ maps to zero, we also use the compatibility of~$M$ with the braiding.
  This time~$d,W \subseteq V$ and the map on the summand~$E(U) \barotimes (R(d) \barotimes E(W))$ factors through the summand~$R(d) \barotimes E(U \sqcup W)$ of~$S^{disc}(X)$.
\end{proof}
To describe~$E$ as an affine-linearly invariant prefactorization algebra on~$\C$, let~$L_0$ denote the grading operator of~$\mathcal{V}$.
This is the endomorphism of~$\mathcal{V}$ defined by~$L_0 a = |a|a$ for~$a$ homogeneous, so that~$\lambda^{L_0}v = \lambda^{|a|}a$ for~$a \in \V$ homogeneous and~$\lambda \in \C^\times$. 
This defines an action of~$\C^\times$ on~$\V$. 
The action of~$\C^\times \ltimes \C$ on~$\V$ is defined in terms of this~$\C^\times$-action via the group homomorphism~$(\lambda,w) \mapsto \lambda$, translations act trivially on~$\V$. 
Therefore~$\mathbf{E}\V$ is an affine-linearly invariant prefactorization algebra where
\[
    \sigma^U_{(\lambda,w)}(\alpha \otimes a_1 \otimes \ldots \otimes a_n) = (\lambda,w)_*(\alpha) \otimes \lambda^{L_0}a_1 \otimes \ldots \otimes \lambda^{L_0}a_n
\]
for~$\alpha \in \mathcal{O}'(U^n \setminus \Delta)$,~$a_1,\ldots, a_n \in \V$, and~$U\subseteq \C$ open.
Here, the pushforward of an analytic functional~$\alpha\in \mathcal{O}'(X)$ along a holomorphic map~$g:X\rightarrow Y$ is 
\[
g_*\alpha := [ f \in \mathcal{O}(Y) \mapsto \alpha(x \mapsto f(g(x)))]\displayperiod  
\]
We now proceed to show that~$\mathbf{F}\V$ is affine-linearly invariant with the maps~$\sigma^U_g$ induced from those of~$\mathbf{E}\V$.
The evaluation maps are natural w.\,r.\,t.\ inclusions because they evaluate a function on~$\C^n \setminus \Delta$ only depending on some vertex algebra elements which are not changed by the extension maps of~$\mathbf{E}\V$.
The evaluation maps take values in~$\Vbarbb$, the set of bounded-below vectors in~$\Vbar$.
Recall that~$\C^\times \ltimes \C$ acts on~$\Vbarbb$
by Proposition~1.4 
of~\cite{GeometricVertexAlgebras}. 
\begin{proposition}\label{proposition:EvaluationIsEquivariant}
  Let~$\V$ be a geometric vertex algebra. 
  The evaluation maps are~$\C^\times \ltimes \C$-equivariant in the sense that the square
  \begin{ctikzcd}
    \mathbf{E}\V(U) \dar{\sigma^U_g} \rar{\ev_U} & \Vbarbb \dar{g}\\
    \mathbf{E}\V(gU) \rar{\ev_{gU}} & \Vbarbb
  \end{ctikzcd}
  commutes for~$U\subseteq \C$ open and~$g \in \C^\times \ltimes \C$. 
  \end{proposition}
  \begin{proof}
  We check equivariance for translation and multiplication maps separately.
  Let
  \[
    [\alpha \otimes a] \in (\mathcal{O}'(U^m\setminus \Delta) \otimes \V^m)_{\Sigma_m} \subseteq \mathbf{E}\V(U)\displayperiod
  \] 
  If~$g \in \C^\times \ltimes \C$, then
  \begin{align*}
    \ev_{g U}(\sigma^U_g([\alpha \otimes a]))&= \ev_{g U}([g_*\alpha \otimes (\lambda.a)]) = (g_*\alpha)(\mu(\lambda.a))\\
    &= \alpha(z \mapsto \mu(g.a)(g.z))\displayperiod
  \end{align*}
  If~$g \in \C^\times \ltimes \C$ is multiplication with some~$\lambda\in \C^\times$, then
  \begin{align*}
  &\ev_{g U}(\sigma^U_g([\alpha \otimes a]))\\
  &= \alpha(z \mapsto \mu(\lambda.a)(\lambda.z))\\
  &= \alpha(z \mapsto \lambda.\mu(a)(z)) &\quad\quad (\C^\times\text{-equivariance})\\
  &= \lambda . \alpha(\mu(a)) &\quad\quad\text{(linearity)}\\
  &= \lambda . \ev_U([\alpha \otimes a])\displaycomma
  \end{align*}
  and thus
  \[
  g.\ev_U([\alpha \otimes a]) = \ev_{gU}(\sigma_g([\alpha \otimes a]))\displayperiod
  \]
  If~$g \in \C^\times \ltimes \C$ is translation by some~$w\in \C$, then
  \begin{align*}
    &\ev_{gU}(\sigma^U_g([\alpha \otimes a]))\\
    &= \alpha(\mathbf{z} \mapsto \mu(a)(\mathbf{z}+w))\\
    &= \alpha(\mathbf{z} \mapsto g.\mu(a)(\mathbf{z})) &\quad\quad \text{(associativity for } m=0\text{)}\\
    &= g.\alpha(\mu(a)) &\quad\quad \text{(linearity)}\\
    &= g. \ev_U([\alpha \otimes a])\displayperiod
  \end{align*}
  Equivariance follows because~$\C^\times \ltimes \C$ is generated by the union of the two subgroups corresponding to~$\C^\times$ and~$\C$.
\end{proof}
\begin{proposition}\label{proposition:FIsInvariant}
  $\mathbf{F}\V$ is an affine-linearly invariant prefactorization algebra on~$\C$ with its invariance maps induced from those of~$\mathbf{E}\V$.  
\end{proposition}
\begin{proof} 
  We have to show that the maps~$\sigma^U_g$ for~$E$ pass to the quotient, their properties then follow.
  The image of a disc under an affine-linear map is again a disc. 
  The induced map on~$E$ preserves membership in~$R$ and hence~$\overline{R^{disc}}$ because~$R(U)$ is the kernel of~$\ev_U$ and the evaluation maps are compatible with the actions of the affine-linear group by Proposition~\ref{proposition:EvaluationIsEquivariant}.
\end{proof}
The rest of this subsection consists of a sequence of propositions proving that~$F$ is holomorphic, i.\,e.,~$F$ is a holomorphic affine-linearly invariant precosheaf on~$\C$.  We first prove this for~$\mathcal{O}'$ using the Cauchy integral formula and the first of these propositions, then proceed to~$\V$, which is a~$G=\C^\times$-representation, then~$E$, and finally~$F$.
\begin{proposition}\label{proposition:OneOverMapIsAnalytic}
  Let~$D\subseteq \C$ be a closed disc.
  The map
  \MapMapsto{D^\circ}{C^0(\partial D)}{w}{\left[z \mapsto \frac{1}{z-w}\right]}
  is holomorphic.
\end{proposition}
\begin{proof}
Let~$w_0 \in D^\circ$ so~$\varepsilon := dist(\partial D, w) > 0$.
Let~$\delta = \varepsilon/2$. 
For~$w \in B_\delta(w_0) \subseteq D^\circ$ and~$z\in\partial D$,
\[
q(z) := \left|\frac{w-w_0}{z-w_0}\right| < 1
\]
and thus
\begin{align*}
\frac{1}{z-w} &= \frac{1}{(z-w_0) - (w - w_0)} = \frac{1}{z-w_0} \frac{1}{1 - \frac{w-w_0}{z-w_0}} \\
              &= \frac{1}{z-w_0} \sum^\infty_{k = 0} \left( \frac{w-w_0}{z-w_0} \right)^k\\
              &=  \sum^\infty_{k=0} \frac{1}{(z-w_0)^{k+1} } (w-w_0)^k\displayperiod
\end{align*}
We now prove uniform convergence of functions of~$z\in \partial D$, that is,
\[
  \left[z \mapsto \frac{1}{z-w}\right] = \sum^\infty_{k=0}  \left[ z\mapsto \frac{1}{(z-w_0)^{k+1}}\right](w-w_0)^k
  \]
in~$C^0(\partial D)$. 
The remainder term
\begin{align*}
\left|\frac{1}{z-w} - \sum^N_{k=0} \frac{1}{(z-w_0)^{k+1}} (w-w_0)^k \right| &= \left| \sum^\infty_{k=N+1}  \frac{1}{(z-w_0)^{k+1}} (w-w_0)^k\right| \\
&\leq \frac{1}{|z-w_0|} \sum^\infty_{k = N + 1} \left| \frac{w-w_0}{z-w_0} \right|^k \\
&\leq \frac{1}{|z-w_0|} \frac{q(z)^{N+1}}{1-q(z)} \\
\end{align*}
converges to zero uniformly in~$z\in \partial D$ because
\[
q(z) = \left|\frac{w-w_0}{z-w_0}\right| \leq \frac{dist(w,w_0)}{dist(\partial D,w_0)} < 1/2 \displayperiod
\]
\end{proof}
\begin{proposition}\label{proposition:DeltaDistributionMapIsAnalytic}
  If~$X\subseteq\C^n$ is open, then
  \labelledMapMapsto{\delta}{X}{\mathcal{O}'(X)}{x}{\delta_x}
  is holomorphic.
\end{proposition}
\begin{proof}
  Let~$x \in X$.
  There is a closed polydisc~$D=D_1 \times \ldots\times D_n \subseteq X$ with center~$x$. 
  Let~$d=d_1 \times \ldots \times d_n$ be an open polydisc concentric with~$D$ s.\,t. the radius of~$d_i$ is strictly than the radius of~$D_i$ for~$i=1,\ldots,n$. 
  We claim that~$\delta(d) \subseteq \mathcal{O}'_D(X)$ and that~$\delta|_d$ is holomorphic as a map to the Banach space~$\mathcal{O}'_D(X)$. 
  Let~$z \in d$ and $f\in \mathcal{O}(X)$.
  We have~$\delta_z \in \mathcal{O}'_D(X)$ since
  \[
    |\delta_z(f)| = |f(z)| \leq ||f||_D \displayperiod
  \]
  The map
  \labelledMapMapsto{F}{d}{\C^0(\partial D_1 \times \ldots \times \ldots \partial D_n)}{w}{\left[z \mapsto \prod^n_{i=1}\frac{1}{z_i-w_i}\right]}
  is holomorphic since the maps
  \begin{align*} 
    d &\overset{\pr_i}{\longrightarrow} d_i \longrightarrow \C^0(\partial D_i)\\
    w &\longmapsto \left[z_i\mapsto \frac{1}{z_i - w_i}\right]
  \end{align*}
  are holomorphic for~$i=1,\ldots,n$ by Proposition~\ref{proposition:OneOverMapIsAnalytic} and the pointwise external product of holomorphic maps is again holomorphic.
  The map
  \labelledMapMapsto{G}{C^0(\partial D_1 \times \ldots \times \partial D_2)}{\mathcal{O}'_D(X)}{f}{\left[g \mapsto \frac{1}{(2\pi i)^n} \int_{\partial D_1 \times \ldots \times \ldots \partial D_n} f(z)g(z) dz\right]}  
  is bounded because
  \[
  \left|\frac{1}{(2\pi i)^n} \int_{\partial D_1 \times \ldots \times \ldots \partial D_n} f(z)g(z) dz\right| \leq ||f|| ||g||_D \displayperiod
  \]
  The map~$\delta|_d$ is analytic because it is equal to~$G\after F$ by the Cauchy integral formula.
\end{proof} 
\begin{proposition}\label{proposition:AnalyticFunctionalsHolomorphicInvariantPrecosheaf}
  Let~$G$ be a complex-analytic group manifold acting on a complex-analytic manifold~$X$ via an analytic map
  \(
    G \times X \rightarrow X.
  \)
  Then the precosheaf~$\mathcal{O}'$ of analytic functionals on~$X$ is a holomorphic~$G$-invariant cosheaf. 
\end{proposition}
\begin{proof}
  Let~$U,V\subseteq X$ be open.
  We wish to show that the action map
  \MapMapsto{\mathcal{D}_{U,V}}{\BVS(\mathcal{O}'(U), \mathcal{O}'(V))}{g}{g_*}
  is holomorphic on the interior~$\mathcal{D}^\circ_{U,V}$ of the domain.
  For each~$g \in \mathcal{D}^\circ_{U,V}$, the map~$g_*$ is the composite of
  \begin{align}
    (u\mapsto (g,u))_*: \mathcal{O}'(U) \rightarrow \mathcal{O}'(\mathcal{D}^\circ_{U,V} \times U) \label{equation:AFunctionInProofOfHolomorphicity}
  \end{align}
  and the bounded, linear pushforward map  
  \[
  \mathcal{O}'(\mathcal{D}^\circ_{U,V} \times U) \longrightarrow \mathcal{O}'(V)
  \]
  along the action map 
  \[
    \mathcal{D}^\circ_{U,V} \times U \longrightarrow V \displayperiod
  \]
  The latter of these is independent of~$g$ and pointwise composition with a fixed bounded linear map preserves holomorphicity, so it suffices to show that the map
 \labelledMapMapsto{I}{\mathcal{D}^\circ_{U,V}}{\BVS(\mathcal{O}'(U), \mathcal{O}'(\mathcal{D}^\circ_{U,V} \times U))}{g}{(u\mapsto (g,u))_*} 
is holomorphic.
Proposition~\ref{proposition:DeltaDistributionMapIsAnalytic} for~$X=\mathcal{D}^\circ_{U,V}$ implies that
\MapMapsto{\mathcal{D}^\circ_{U,V}}{\mathcal{O}'(\mathcal{D}^\circ_{U,V})}{g}{\delta_g}
is holomorphic.
It follows that 
\MapMapsto{\mathcal{D}^\circ_{U,V}}{\BVS(\mathcal{O}'(U),\mathcal{O}'(\mathcal{D}^\circ_{U,V}))\barotimes\mathcal{O}'(U))}{g}{[\alpha \mapsto \delta_g \barotimes \alpha]}
is holomorphic because the map
\MapMapsto{\mathcal{O}'(U)}{\mathcal{O}'(\mathcal{D}^\circ_{U,V})\barotimes\mathcal{O}'(U))}{\alpha}{\delta_g \barotimes \alpha}
is bounded and linear. 
Postcomposing with the external product of functionals
\[
\times : \mathcal{O}'(\mathcal{D}^\circ_{U,V})\barotimes \mathcal{O}'(U) \longrightarrow  \mathcal{O}'(\mathcal{D}^\circ_{U,V} \times U)
\]
preserves holomorphicity.
Overall, the map
\labelledMapMapsto{J}{\mathcal{D}^\circ_{U,V}}{\mathcal{O}'(\mathcal{D}^\circ_{U,V} \times U)}{g}{\delta_g \times \alpha}
is holomorphic. 
It remains to identify~$I$ with~$J$.
Let~$g \in \mathcal{D}^\circ_{U,V}$.
It suffices to show that~$I(g)$ and~$J(g)$ agree on the dense subset of product functions~$p(z,w) = f(z)h(w), (z,w)\in \mathcal{D}^\circ_{U,V}$ for~$f \in \mathcal{O}(\mathcal{D}^\circ_{U,V})$ and~$h \in \mathcal{O}(U)$:
\[
J(g)(p) = \delta_g( z\mapsto \alpha(w\mapsto f(z)h(w))) = \delta_g(f) \alpha(h) = f(g) \alpha(h)\displaycomma
\]
which coincides with
\begin{align*}
I(g)(p) &= \left((u \mapsto (g,u))_*(\alpha)\right)(p) =  \alpha( (u \mapsto (g,u))^*p) =\alpha( u \mapsto p(g,u)) \\
&= \alpha( u \mapsto f(g)h(u)) = f(g) \alpha(h)\displayperiod
\end{align*}
\end{proof}
\begin{proposition}\label{proposition:TensorProductsOfHolomorphicalInvariantPrecosheaves}
  If~$A$ and~$B$ are holomorphic invariant precosheaves, then their pointwise completed tensor product~$A \barotimes B$ is a holomorphic invariant precosheaf in a natural way.
\end{proposition}
\begin{proof}
  The function~$\rho^{A\barotimes B}_{U,V}$ is holomorphic because
  \[
    \barotimes : \BVS(A(U), A(V)) \barotimes \BVS(B(U), B(V)) \\
    \rightarrow \BVS(A(U) \barotimes B(U), A(V) \barotimes B(V))
  \]
  is a bounded linear map so postcomposing with it preserves holomorphicity.
\end{proof} 
\begin{proposition}\label{proposition:QuotientOfHolomorphicPrecosheafIsHolomorphic}
  Let~$F_1,F_2$ be affine-linearly invariant precosheaves on~$\C$ of complete bornological vector spaces. 
  If~$F_2$ is a quotient of~$F_1$, the invariance isomorphisms of~$F_1$ are induced from those of~$F_1$, and~$F_1$ is holomorphic, then~$F_1$ is holomorphic. 
\end{proposition}
Here, quotient means cokernel by some map of precosheaves of complete bornological spaces.  
\begin{proof}
  Let~$X,Y$ be complete bornological spaces and~$A \subseteq X$ a sub vector space. The map
  \begin{align}
    \{ f \in \BVS(X,Y) \mid f|_A = 0 \} \longrightarrow \BVS(X/\overline{A}, Y)\label{display:TakingInducedMapsOnQuotientIsBounded}
  \end{align}
  sending~$f$ to the map induced by~$f$ is a bounded linear map.
  Let~$i: G \hookrightarrow F_1$ be a pointwise inclusion of precosheaves of which~$q: F_1 \rightarrow F_2$ is a cokernel.
  For~$U,V \subseteq \C$ open, we apply this to~$X=F_1(U)$,~$A = R(U)$, and~$Y = F_2(V)$ to conclude that the action map~$\rho^{F_2}_{U,V}$ for~$F_2$ is holomorphic because it is the composite of~$q_* \after \rho^{F_1}_{U,V}$ with the bounded linear map from~(\ref{display:TakingInducedMapsOnQuotientIsBounded}).  
\end{proof}
\begin{proposition}\label{proposition:DirectSumsOfHolomorphicPrecosheavesAreHolomorphic}
  Direct sums of holomorphic precosheaves are again holomorphic.
\end{proposition}
\begin{proof}
This follows from the compatibility of direct sums with the enrichment of~$\CBVS$ over itself and Proposition~\ref{proposition:HolomorphicIfAndOnlyIfComponentsAre} that a product-valued function is holomorphic if and only if its components are.  
\end{proof}
The same argument proves that direct sums of holomorphic representations are holomorphic, so~$\V$ is holomorphic as the direct sum of holomorphic one-dimensional representations.
Thus the constant precosheaf assigning~$\V$ to every open is holomorphic.
\begin{proposition}\label{proposition:EandFareHolomorphic}
  The precosheaves~$E$ and~$F$ on~$\C$ are holomorphic.
\end{proposition}
\begin{proof}
  The preceding propositions imply that, for~$n \geq 0$ the precosheaves given by 
  \(
  \mathcal{O}'(U^n \setminus \Delta) \otimes \V^n   
  \)
  on~$U\subseteq \C$ are holomorphic.
  Proposition~\ref{proposition:QuotientOfHolomorphicPrecosheafIsHolomorphic} and Proposition~\ref{proposition:DirectSumsOfHolomorphicPrecosheavesAreHolomorphic} about quotients resp.\ sums imply that~$E$ is holomorphic. Proposition~\ref{proposition:QuotientOfHolomorphicPrecosheafIsHolomorphic} implies that~$F$ is holomorphic.
\end{proof}

\Subsection{Back to Geometric Vertex Algebras}\label{subsection:BackToGeometricVertexAlgebras}
Having obtained a holomorphic affine-linearly invariant prefactorization algebra~$\mathbf{F}\V$ from a geometric vertex algebra~$\V$ in the previous subsection, we now show that we can recover~$\V$ using the procedure of Section~\ref{subsection:FromPrefactorizationAlgebrasToGeometricVertexAlgebras}.
First, we identify~$\mathbf{V}\mathbf{F}\V$ with~$\V$ as a~$\Z$-graded bornological vector space, proving the discreteness of~$\mathbf{V}\mathbf{F}\V$.
Second, we identity the multiplication maps of~$\mathbf{V}\mathbf{F}\V$ with those of~$\V$. 
This implies that~$\mathbf{F}\V$ has meromorphic OPE.
In summary, this means that~$\mathbf{F}\V$ is the sort of prefactorization algebra that gives rise to a (geometric) vertex algebra, and we show that this geometric vertex algebra~$\mathbf{V}\mathbf{F}\V$ is isomorphic to~$\V$. 

We will use the evaluation map~$\ev_U$ from~$E(U)$ to~$\Vbar$ for~$U$ a disc to get a map~$\mathbf{V}\mathbf{F}\V \rightarrow \Vbar$ which will turn out to have image in~$\V$.
The following lemma implies that~$\ev_U$ induces a linear map~$F(U) \rightarrow \Vbar$.
Its proof uses the associativity of~$\V$ which we have not used so far, except for the case~$m=0$ which is translation invariance.
\begin{lemma}\label{lemma:RelationsOnDisc}
If~$D \subset \C$ is an open disc or~$U= \C$, then~$R(D) = R^{disc}(D)$.
\end{lemma}
\begin{proof}
First, let~$D \subseteq \C$ be an open disc.
The inclusion~$R(D)\subseteq R^{disc}(D)$ follows from the definition by considering~$d=D$ and~$W=\emptyset$ as in Definition~\ref{definition:F}, that is, all~$x \in R(D)$ are equal to the product of~$x$ with~$1 \in E(\emptyset)$. 
For the other inclusion, let~$d \subseteq D$ be an open disc and~$W\subseteq D$ be open and disjoint from~$d$.
Our goal is to show that~$M^{d,W}_U(R(d) \otimes E(W)) \subseteq R(D)$ because sets of this form span~$R^{disc}(D)$.
We may restrict attention to~$r \otimes e \in R(d) \otimes E(W)$ with~$r = \sum^m_{i=0} \alpha^i\otimes a^i$ with~$\alpha^i \in \mathcal{O}'(d^i \setminus \Delta)$ for some~$m\geq 0$ and~$a^i \in \V^{\otimes i}$ and~$e=\beta \otimes b$ with~$\beta \in \mathcal{O}'(W^n \setminus \Delta)$ and~$b \in \V^{\otimes n}$ for some~$n \geq 0$.
Let~$p$ be the center of~$d$. 
Since~$r$ is a relation, so is~$\sigma^d_{-p}(r)$, which means that
\begin{align}
  0 = \ev(\sigma^d_{-p}(r)) = \sum^m_{i=0} ((-p)_\ast\alpha^i)(\mu(a^i)) = \sum^m_{i=0} \alpha^i[ \mathbf{z} \mapsto \mu(a^i)(\mathbf{z} - p)]\displaycomma \label{display:ShiftedRelation}
\end{align}
and this is what we need to prove that~$M^{d,W}_D(r\otimes e)$ is a relation by using the associativity of~$\V$. 
We denote the first~$i$ coordinates of~$D^{i+n} \subseteq \C^{i+n}$ by~$\mathbf{z}$ and the last~$n$ by~$\mathbf{w}$. 
\begin{align*}
&\ev_D(M^{d,W}_{D}(r\otimes e)) \\
&= \sum^m_{i=0} \alpha^i\beta[(\mathbf{z},\mathbf{w})
\mapsto \mu(a^i \otimes b)(z,w)] \\
&= \sum^m_{i=0} \beta[\mathbf{w} \mapsto \alpha^i[ \mathbf{z} \mapsto \mu(a^i \otimes b)(z,w)]] \\
&= \sum^m_{i=0} \beta[\mathbf{w} \mapsto \alpha^i\left(\sum_{k \in \Z}[ \mathbf{z} \mapsto  \mu(p_k\left(\mu(a^i)(\mathbf{z}-p)\right)\otimes \mathbf{b})(p,\mathbf{w})]\right)] \\
\intertext{
(because of associativity, which holds for tensors like~$a^i$, not just elementary tensors)}
&= \sum^m_{i=0} \beta[\mathbf{w} \mapsto \sum_{k \in \Z}\alpha^i[ \mathbf{z} \mapsto  \mu(p_k\left(\mu(a^i)(\mathbf{z}-p)\right)\otimes \mathbf{b})(p,\mathbf{w})]] \\
\intertext{
(because of the continuity of $\alpha^i$)}
&= \sum^m_{i=0}\beta[\mathbf{w}
\mapsto \sum_{k \in \Z} \mu(p_k\left(\alpha^i [\mathbf{z} \mapsto \mu(a^i)(\mathbf{z}-p)]\right)\otimes \mathbf{b})(p,\mathbf{w}) ]\\ 
&= \sum^m_{i=0}\beta[\mathbf{w}
\mapsto \sum_{k \in \Z} \mu(\alpha^i[\mathbf{z} \mapsto p_k\left(  \mu(a^i)(\mathbf{z}-p)\right)]\otimes \mathbf{b})(p,\mathbf{w}) ]\\ 
\intertext{(because evaluating analytic functionals on vector-valued functions is compatible with linear maps)}
&= \beta[\mathbf{w}
\mapsto \sum_{k \in \Z} \mu(p_k\left(\sum^m_{i=0}\alpha^i [\mathbf{z} \mapsto \mu(a^i)(\mathbf{z}-p)]\right)\otimes \mathbf{b})(p,\mathbf{w}) ]\\ 
\intertext{(because of the componentwise definition of the evaluation of analytic functionals on functions with values in a~$\Z$-graded vector space)}
&= \beta[\mathbf{w}
\mapsto \sum_{k \in \Z} \mu(p_k\left(\sum^m_{i=0}\alpha^i [\mathbf{z} \mapsto \mu(a^i)(\mathbf{z}-p)]\right)\otimes \mathbf{b})(p,\mathbf{w}) ]\\ 
  &=\beta[ \mathbf{w}\mapsto  \sum_{k \in \Z} \mu(0 \otimes b)(p,\mathbf{w})]\text{\quad\quad\quad\quad(see Equation~(\ref{display:ShiftedRelation}))}\\ 
  &= 0
\end{align*}
The case of~$U=\C$ follows because every analytic functional on~$\C^n$ arises from some bounded subset, so~$U=\C$ reduces to the case of~$U$ some sufficiently large disc.  
\end{proof}
\begin{corollary}
For every open subset~$U\subseteq \C$, the evaluation map~$\ev_U$ factors through~$F(U)$. 
\end{corollary}
\begin{proof}
By Lemma~\ref{lemma:RelationsOnDisc}, this is the case for~$U=\C$ and~$\ev_U = \ev_\C \after E(U \subseteq \C)$.  
\end{proof}
We denote the induced map~$F(U) \rightarrow \Vbar$ by $\ev_U$, too.
\begin{remark}
Another consequence of the lemma is that~$F(D) = E(D)/R(D)$ for~$D$ a disc since we know~$R(D)$ to be b-closed; see Remark~\ref{remark:ClosednessOfRelations}.
By the definition of~$R(D)$, this means that we can identify~$F(D)$ with a sub vector space of~$\Vbar$.
\end{remark}
The comparison map~$I: \V\rightarrow \mathbf{V}\mathbf{F}\V$ is defined as follows.
For every~$\varepsilon>0$, evaluation at zero is an element~$\delta_0 \in \mathcal{O}'(B_\varepsilon(0))$, which is fixed under the action of~$\D^\times$. For homogeneous~$a \in \V_k$, we define~$I(a) \in \mathbf{V}\mathbf{F}\V$ to be the element corresponding to
\[
  [\delta_0 \otimes a] \in F(B_\varepsilon(0))_k \cong F(0)_k = (\mathbf{V}\mathbf{F}\V)_k \displayperiod
\]
It is clear that this is a linear map on the summands of~$\V$, and we extend~$I$ to~$\V$ by linearity. 
In other words, $I$ is determined by the formula~$P_R(I(a)) = [\delta_0 \otimes a]$ where we recall that~$P_R$ is the natural map~$\mathbf{V}\mathbf{F}\V \rightarrow \mathbf{F}\V(B_R(0))$ for~$R>0$.
\begin{proposition}\label{proposition:VFisIdAsGradedBVS}
The linear map
\labelledMapMapsto{I}{\V}{\mathbf{V}\mathbf{F}\V}{a}{[\delta_0 \otimes a]\quad\quad (\mathrm{for}\ a\ \mathrm{homogeneous})}
is an isomorphism of~$\Z$-graded complete bornological spaces.
\end{proposition}
\begin{proof}
  It is clear that~$I$ is bounded since~$\V$ is discrete. 
  For all~$r>0$, let~$J_r$ be the composite of~$P_r : \mathbf{V}\mathbf{F}\V \rightarrow \mathbf{F}\V(B_r(0))$ from Section~\ref{subsection:FromPrefactorizationAlgebrasToGeometricVertexAlgebras} (for~$F := \mathbf{F}\V$) with the evaluation map~$\mathbf{F}\V(B_r(0)) \rightarrow \Vbar$.
  The evaluation map, which takes values in~$\Vbar$ in general, maps elements arising from~$\mathbf{V}\mathbf{F}\V$ to~$\V$ because the evaluation map is~$\C^\times$-equivariant.
  We can thus view~$J_r$ as a map to~$\V$. 
  The insertion-at-zero axiom implies that 
  \[
    J_r(I(a)) = \ev_{B_R(0)}(P_R(I(a))) = \ev_{B_R(0)}([\delta_0 \otimes a]) = \mu(a,0) = a
  \]
  for homogeneous~$a$.
  This proves that~$J_r \after I = \id_\V$ because every element of~$\V$ is a sum of homogeneous elements.   
  
  It remains to prove that~$J_r$ is injective and bounded.
  It is a bounded linear map as the composite of two bounded linear maps. 
  The inclusion~$\mathbf{V}F \hookrightarrow \overline{\mathbf{V}F}$ is the composite of~$P_R$ with the map~$L^R : F(B_R(0)) \rightarrow \overline{\mathbf{V}F}$, so~$P_R$ is injective.
  Lemma~\ref{lemma:RelationsOnDisc} says that~$R^{disc}(B_r(0)) = R(B_R(0))$, so the evaluation map~$F(B_R(0)) \rightarrow \Vbar$ is injective. Thus~$J_r = \ev_{B_R(0)} \after P_R$ is injective.
\end{proof}
The previous proposition implies that~$\mathbf{V}\mathbf{F}\V$ is discrete.
The comparison map~$I$ induces a map~$\overline{I}: \Vbar \rightarrow \overline{\mathbf{V}\mathbf{F}\V}$ on the products of the weight spaces.
Let~$\widetilde{\mu}$ denote the multiplication maps of~$\mathbf{V}\mathbf{F}\V$. 
\begin{proposition}\label{proposition:VFisIdAsGeometricVertexAlgebras}
  $I$ is a homomorphism w.\,r.\,t.\ the geometric vertex algebra structures, that is,
  for~$m\geq 0$, and~$z \in \C^m \setminus \Delta$, the square
  \begin{ctikzcd}
    \V^{\otimes m} \rar{I^{\otimes m}} \dar{\mu^z} & \mathbf{V}\mathbf{F}\V^{\otimes m}\dar{\widetilde{\mu}^z} \\
    \Vbar \rar{\overline{I}} & \overline{\mathbf{V}\mathbf{F}\V}
  \end{ctikzcd} 
  commutes.
  
  It follows that~$\mathbf{V}\mathbf{F}\V$ satisfies the meromorphicity axiom; hence~$\mathbf{F}\V$ has meromorphic OPE,~$\mathbf{V}\mathbf{F}\V$ is a geometric vertex algebra, and~$I$ is an isomorphism of geometric vertex algebras \[\V \cong \mathbf{V}\mathbf{F}\V \displayperiod \]
\end{proposition}
\begin{proof}
  Let~$k \in \Z$ and~$R>0$ be big enough so that~$z_1,\ldots, z_m \in B_R(0)$.
  The image in~$F(B_R(0))$ of the~$k$-th component of the lower composition evaluated at~$a \in \V^{\otimes m}$ is 
\begin{align*}
  P_R(p_k(\overline{I}(\mu(a)(z)))) &= P_R(I(p_k(\mu(a)(z)))) = \left[\delta_0 \otimes p_k(\mu(a)(z))\right]\displaycomma\\
\end{align*}
and this evaluates to~$p_k(\mu(a)(z))$. 
For the upper composition and~$a$ an elementary tensor~$a_1 \otimes \ldots \otimes a_m \in \V^{\otimes m}$ with~$a_i \in V$ homogeneous, we get
\begin{align*}
  P_R(p_k(\widetilde{\mu}(I^{\otimes m}(a))(z))) &= P_R(p_k(\widetilde{\mu}([\delta_0 \otimes a_1] \otimes \ldots \otimes [\delta_0 \otimes a_m])(z)))\\
  &= P_R(p_k(\widetilde{\mu}([\delta_{(z_1,\ldots,z_m)} \otimes a_1 \otimes \ldots \otimes a_m])))\\
  &= P_R(p_k(\widetilde{\mu}([\delta_z \otimes a])))\\
  &= l^R_k([\delta_z \otimes a])\\
  &= \ootpi \oint w^{-k-1} w.[\delta_z \otimes a]dw 
\end{align*}
and this evaluates to
\begin{align*}
  &\ootpi \oint w^{-k-1} w.\ev_{B_R(0)}([\delta_z \otimes a])dw 
  \intertext{because evaluation is equivariant. This equals}
  &\ootpi \oint w^{-k-1} w.\mu(a,z)dw= p_k(\mu(a,z))
\end{align*}
in~$\V_k$.
Two elements of~$\mathbf{V}\mathbf{F}\V_k$ which evaluate to the same element of~$\V_k$ must be equal, so the square commutes and~$I$ is a homomorphism.

The meromorphicity axiom for~$(\V, \mu)$ implies the meromorphicity axiom for~$(\mathbf{V}\mathbf{F}\V, \widetilde{\mu})$ since~$I$ and hence~$I^{\otimes m}$ and~$\overline{I}$ are isomorphisms of vector spaces by Proposition~\ref{proposition:VFisIdAsGradedBVS}.
\end{proof} 
\begin{remark}
  It can be shown that the image of~$\V$ in~$\mathbf{F}\V(B_r(0))$ is dense for all~$0< r \leq \infty$.
  Furthermore, for~$0<r\leq R \leq \infty$, the maps
  \[
  \V \hookrightarrow \mathbf{F}\V(B_r(0)) \hookrightarrow \mathbf{F}\V(B_R(0)) \hookrightarrow \Vbar  
  \] 
  are embeddings of vector spaces, all of which are complete bornological spaces.
\end{remark}

%% file: fact.tex
\Section{Factorization Algebras}\label{section:FactorizationAlgebras}
In this section, we prove that~$\mathbf{F}\V$ is a factorization algebra with values in the symmetric monoidal category of complete bornological spaces. 
We start by recalling the definition of factorization algebras.
\begin{definition}
  A prefactorization algebra~$F$ on a topological space~$X$ with values in a symmetric monoidal category~$\mathcal{C}$ is called \emph{multiplicative} if the map
  \[
    M^{U,V}_{U\sqcup V}: F(U) \otimes F(V) \rightarrow F(U \sqcup V)
  \]
  is an isomorphism for all disjoint opens~$U,V \subseteq X$ and the unit map~$\mathbf{1}_\mathcal{C} \rightarrow F(\emptyset)$ is an isomorphism.
\end{definition}
In our context,~$\otimes$ is the completed tensor product~$\barotimes$ of complete bornological vector spaces defined in\ref{definition:CompletedTensorProduct}.

Recall that a presheaf~$G$ on a topological space~$X$ is called a \emph{sheaf} if it satisfies descent along open covers, that is, if
\[
G(Y) \rightarrow \prod_{U \in \mathcal{U}} G(U) \rightrightarrows \prod_{U,V\in\mathcal{U}} G(U \cap V)
\]
is a limit diagram for every open cover~$\mathcal{U}$ of every open subset~$Y$ of~$X$.
Similarly, a precosheaf~$F$ on~$X$ is called a \emph{cosheaf} if it satisfies codescent along open covers, that is, if the diagram 
\begin{align}
\bigoplus_{U,V\in\mathcal{U}} F(U \cap V) \rightrightarrows \bigoplus_{U\in \mathcal{U}} F(U) \rightarrow F(Y) \label{equation:Cosheaf}
\end{align}
is a colimit diagram. 
In the definition of a factorization algebra, we demand codescent for special open covers called Weiss covers. 
\begin{definition}
  A family~$\mathcal{U}$ of open subsets of a topological space~$X$ is called a \emph{Weiss cover} if, for every finite subset~$F\subseteq X$, there is a~$U\in\mathcal{U}$ with~$F \subseteq U$.  
  A precosheaf on a topological space~$X$ is a~\emph{Weiss cosheaf} if diagram~(\ref{equation:Cosheaf}) is a colimit diagram for all open subsets~$Y \subseteq X$ and Weiss covers~$\mathcal{U}$ of~$Y$.
\end{definition}
Every Weiss cover is an open cover.
A family~$\mathcal{U}$ of subsets of a space~$X$ is a Weiss cover if and only if~$\{U^n\mid U \in \mathcal{U}\}$ is an open cover of~$X^n$ for all~$n \geq 0$. 
Another way of phrasing the (Weiss) cosheaf condition is the following, if the category in which the precosheaf~$F$ takes values has all colimits: For~$Y \subseteq X$ open and all (Weiss) open covers~$\mathcal{U}$ of~$Y$ closed under taking pairwise intersections, the natural map
\[
  \colim_{U \in \mathcal{U}} F(U) \longrightarrow F(X)
\]
is an isomorphism, where~$\mathcal{U}$ considered as a poset under inclusions of subsets of~$Y$, and the colimit is taken over the extension maps of~$F$.

\begin{definition}
A \emph{factorization algebra} is a prefactorization algebra which is multiplicative and whose underlying precosheaf is a Weiss cosheaf.   
\end{definition} 
We fix a geometric vertex algebra~$\V$ and use the abbreviations~$E=\mathbf{E}\V$ and~$F=\mathbf{F}\V$. 
\RelatedWork{}
Factorization algebras, as considered here, were introduced by Costello and Gwilliam and the main reference is~\cite{CostelloGwilliam} (see also~\cite{GwilliamThesis}).
We include multiplicativity in the definition of factorization algebra; in~\cite{CostelloGwilliam} multiplicativity is a property of some factorization algebras.
Earlier, Beilinson and Drinfeld introduced factorization algebras in algebraic geometry~\cite{BeilinsonDrinfeldChiralAlgebras}.
The factorization algebras we have in mind are typically not locally constant, since the vertex algebra~$\mathbf{V}F$ arises from a commutative algebra if~$F$ is locally constant, more specifically a~$\Z$-graded commutative algebra with derivation of degree~$1$.  
The works~\cite{FrancisGaitsgoryChiralKoszulDuality},~\cite{LurieHA},~\cite{GinotNotes},~\cite{MatsuokaDescent},~\cite{KnudsenHigherEnvelopingAlgebras}, and~\cite{LejayThesis} contain, among other things, material relevant to locally constant factorization algebras and~$E_n$-algebras, generalizing to higher dimensions and target categories like the~$\infty$-category of chain complexes.
See~\cite{GwilliamRejzner} and~\cite{BeniniPerinSchenkel} for the relationship between factorization algebras and algebraic quantum field theory.

\Subsection{More Functional-Analytic Preliminaries}
There are two more topics in functional analysis which we did not treat in our summary of bornological spaces in Section~\ref{subsection:FunctionalAnalyticPrelim} and reference in the proofs of multiplicativity of~$F$ and Weiss codescent.
The first is nuclearity, the bornological vector space~$\mathcal{O}'(U)$ of analytic functionals on an open~$U \subseteq \C^n$ is nuclear.
References for nuclear spaces include~\cite{Treves} and~\cite{HogbeNlendMoscatelli}, the latter includes the bornological setting. 
The second is the generalization of the identity theorem for holomorphic functions with values in a Banach space to the case of a complete bornological space. 

We call a completely normable bornological space~\emph{Hilbertable} if one of the complete norms inducing its bornology is a Hilbert space norm. 
The notion of Hilbert-Schmidt map between Hilbert spaces only depends on the bornologies of source and target. 
A complete bornological space is~\emph{nuclear} if every bounded subset~$B$ is bounded in a Hilbertable subobject~$H_1$ so that~$H_1$ is contained in a Hilbertable subobject~$H_2$ such that the inclusion~$H_1 \subseteq H_2$ is Hilbert-Schmidt. 
Let~$\Hilb$ denote the category of Hilbert spaces and bounded linear maps.
The functor from~$\Hilb$ to~$\CBVS$ sending a Hilbert space to its underlying complete bornological space is a full embedding with image the Hilbertable spaces.
In~\cite[Chapter III]{HogbeNlendMoscatelli}, a locally convex space~$X$ is defined to be~\emph{nuclear} if~$X'$ is a nuclear bornological space, and this definition is shown to equivalent to some of the usual ones, e.\,g., in terms of Hilbert-Schmidt maps.  
Since~$\mathcal{O}(X)$ is nuclear for~$X \subseteq \C^n$ open, it follows that~$\mathcal{O}'(X)$ is a nuclear bornological space.  
Our next goal is to prove that~$\mathcal{O}'$ is a cosheaf for Stein open covers of Stein subsets of~$\C^n$. 
A Stein open cover is an open cover in which every open set is Stein. 
The intersection of two Stein open subsets is again Stein. 
For~$X \subseteq \C^n$ open, let~$\mathcal{O}_{L^2}(X)$ denote the Hilbert space of holomorphic~$L^2$-functions on~$X$.
\begin{proposition}
  The presheaf~$\mathcal{O}_{L^2}$ of Hilbert spaces satisfies the sheaf condition for finite covers.  
\end{proposition}
\begin{proof}
  Finite products and kernels exist in the category of Hilbert spaces and bounded maps, and their underlying vector spaces are products resp. kernels in the category of vector spaces. 
  Let~$V \subseteq \C^m$ be open and~$V_1, \ldots, V_n$ cover~$V$.
If~$f \in \mathcal{O}(V)$ denotes the result of gluing~$L^2$-functions on the~$V_i$, then
  \[
  ||f||^2_{L^2(V)} \leq  ||f|_{V_1}||^2_{L^2(V_1)} + \ldots + ||f|_{V_n}||^2_{L^2(V_n)} \displaycomma
\]
so~$\mathcal{O}_{L^2}$ is a sheaf of vector spaces. 
This inequality also proves that the gluing map is bounded.
\end{proof}
\begin{corollary}\label{corollary:L2AnalyticFunctionalsCosheafForFiniteCovers}
  The precosheaf~$\mathcal{O}'_{L^2}$ of Hilbert spaces satisfies the cosheaf condition for finite covers.  
\end{corollary}
\begin{proof}
  Taking the dual is a functor~$\Hilb \rightarrow \Hilb^{\op}$ which preserves finite limits because it is an equivalence of categories.
\end{proof}
If~$U$ and~$X$ are open subsets of~$\C^n$, we use the notation~$U \subset\subset X$ to mean that~$\overline{U} \subseteq X$ and~$\overline{U}$ is compact.  
\begin{proposition}\label{proposition:AnalyticFunctionalOnSteinColimitOfL2s}
  If an open~$X \subseteq \C^n$ is Stein, then
  \[
    \mathcal{O}'(X) \cong \colim_{U \subset\subset X} \mathcal{O}'_{L^2}(U)
  \]
  in the category of bornological spaces. 
\end{proposition}
\begin{proof}
  The comparison map is 
  \labelledMapMapsto{\Phi}{\colim_{U \subset\subset X} \mathcal{O}'_{L^2}(U)}{\mathcal{O}'(X)}{\alpha}{[f \mapsto \alpha(f|_U)]}
  whose well-definedness uses the fact that supremum norm~$||\ ||_{\infty, \overline{U}}$ of continuous functions on~$\overline{U}$ bounds~$||\ ||_{L^2(U)}$ from above.
  We prove that~$\Phi$ is surjective.
  Let~$\alpha \in \mathcal{O}'(X)$. There is a compact set~$K \subset X$ s.\,t.~$\alpha$ is continuous as a map from~$(\mathcal{O}(X), ||\ ||_{\infty, K})$ to~$\C$.
  Let~$U\subset\subset X$ be a neighborhood of~$K$. 
  Let~$\res: (\mathcal{O}(X), ||\ ||_{\infty, K})\rightarrow (\mathcal{O}_{L^2}(U),||\ ||_{\infty, K})$ denote the restriction map.
  It is continuous. 
  Then
  \[
    \widetilde{\alpha} = \alpha \after \res^{-1} : \im \res \rightarrow \C
  \]
  is continuous, in particular well-defined. 
  By the Hahn-Banach theorem,~$\widetilde{\alpha}$ extends to a continuous linear functional~$\beta$ defined on~$(\mathcal{O}_{L^2}(U),||\ ||_{\infty, K})$ with operator norm~$||\beta|| = ||\widetilde{\alpha}|| = ||\alpha||$. 
  The map 
  \[
    i : (\mathcal{O}_{L^2}(U),||\ ||_{L^2(U)}) \rightarrow (\mathcal{O}_{L^2}(U),||\ ||_{\infty, K})
  \]
  given by the identity map on elements is continuous as a consequence of the Cauchy integral formula. 
  Then the class of~$\beta \after i$ is a preimage of~$\alpha$ under~$\Phi$.

  We prove that~$\Phi$ is injective. 
  Let~$\alpha \in \mathcal{O}'_{L^2}(U),\ U \subset\subset X$, represent a class in~$\ker \Phi$.
  This means that~$\alpha(f|_U)=0$ for all~$f \in \mathcal{O}(X)$.
  Without loss of generality, we may assume that~$\overline{U}$ is~$\mathcal{O}(X)$-convex, since the~$\mathcal{O}(X)$-convex hull of a compact set in a Stein manifold is again compact. 
  Let~$V\subset\subset X$ be a neighborhood of~$\overline{U}$. 
  It suffices to see that~$\alpha(g|_U) = 0$ for all~$g \in \mathcal{O}_{L^2}(V)$.
  Let~$A(g) = \alpha(g|_{U})$ for~$g \in \mathcal{O}_{L^2}(V)$. 
  We now argue that the map~$A : \mathcal{O}_{L^2}(V) \rightarrow \C$ is continuous w.\,r.\,t.\ the semi-norm~$||\ ||_{\infty,\overline{U}}$ on the source.
  First, the semi-norm~$||\ ||_{\infty, \overline{U}}$ is at least as large as the~$L^2$-norm on~$U$ with respect to which~$\alpha$ is continuous by assumption. 
  Second, recall that the inclusion~$\mathcal{O}_{L^2}(V) \hookrightarrow \mathcal{O}(V)$ is continuous.
  In particular, up to a constant factor, the semi-norm~$||\ ||_{L^2(V)}$ is at least as large as the continuous semi-norm~$||\ ||_{\infty, \overline{U}}$ on~$\mathcal{O}(V)$.
  This concludes the argument that~$A$ is continuous for the semi-norm~$||\ ||_{\infty,\overline{U}}$.
  This is the semi-norm of uniform convergence on~$\overline{U}$.
  Since~$\overline{U}$ is~$\mathcal{O}(X)$-convex, every holomorphic function on a neighborhood of~$\overline{U}$ can be uniformly approximated on~$\overline{U}$ by holomorphic functions on~$X$; see~\cite[Chapter VII, A, Theorem 6]{GunningRossi}.
  The restriction to~$V$ of a holomorphic function on~$X$ is a holomorphic~$L^2$-function on~$V$ because~$\overline{V}$ is compact. 
  By assumption,~$\alpha$ and thus~$A$ vanish on functions which extend holomorphically to~$X$. 
  The set of such functions is~$||\ ||_{\infty,\overline{U}}$-dense in~$\mathcal{O}_{L^2}(V)$, the source of~$A$, so~$A=0$. 
  Hence the class of~$\alpha$ is zero in the colimit.

  It remains to prove that~$\Phi^{-1}$ is bounded. 
  Let~$B \subseteq \mathcal{O}'(X)$ be bounded.
  There is a compact~$K \subseteq X$ and a~$C>0$ with~$||\alpha|| \leq C$ for all~$\alpha \in B$ where the operator norm is taken w.\,r.\,t.~$||\ ||_{\infty,K}$ on~$\mathcal{O}(X)$.
  Let~$U\subset\subset X$ be a neighborhood of~$K$.
  Then~$\Phi^{-1}(\alpha)$ is represented by~$\beta \after i \in \mathcal{O}'_{L^2}(U)$, with
  \[
    ||\beta \after i || \leq ||\beta|| ||i|| = ||\alpha|| ||i|| \leq C ||r||\displayperiod
  \]
  Since~$C||r||$ is independent of~$\alpha$, so~$\Phi^{-1}(\alpha)$ contained in the image of a bounded subset in~$\mathcal{O}'_{L^2}(U)$, namely the closed ball of radius~$C ||i||$. 
\end{proof}
\begin{proposition}\label{proposition:AnalyticFunctionalsFormACosheafForSteinCoversOfSteinSubsetsOfCn}
  The precosheaf~$\mathcal{O}'$ of complete bornological spaces satisfies the cosheaf condition for Stein open covers of Stein subsets of~$\C^n$.
\end{proposition}
\begin{proof}
  Let~$Y$ be a Stein subset of~$\C^n$ and~$\mathcal{U}$ be a Stein open cover of~$Y$ closed under pairwise intersections.
  Corollary~\ref{corollary:L2AnalyticFunctionalsCosheafForFiniteCovers} implies that~$\mathcal{O}'_{L^2}$ satisfies the cosheaf property in the category of complete bornological spaces because the inclusion~$\Hilb \hookrightarrow \CBVS$ preserves finite colimits. 
  In the following,~$\mathcal{F}$ ranges over all finite sets of open subsets of~$Y$ closed under intersection such that, for every~$V \in \mathcal{F}$, there is a~$U \in \mathcal{U}$ with~$V \subset\subset U$. 
  The collection of such~$\mathcal{F}$ is a partially ordered set w.\,r.\,t.\ inclusions. It is filtered because~$\mathcal{F}$ and~$\mathcal{F}'$ are both contained in the closure of~$\mathcal{F}$ and~$\mathcal{F}'$ under pairwise intersections. 
  The map~$\colim_{U \in \mathcal{U}} \mathcal{O}'(U) \rightarrow \mathcal{O}'(Y)$ factors as 
  \begin{align}
    \colim_{U \in \mathcal{U}} \mathcal{O}'(U) &\cong \colim_{U \in \mathcal{U}} \colim_{V \subset\subset U}\mathcal{O}'_{L^2}(V)&\text{(Proposition~\ref{proposition:AnalyticFunctionalOnSteinColimitOfL2s})}\label{display:AnalyticFunctionalsCosheaf:first}\\ 
    &\cong \colim_{\mathcal{F}} \colim_{V \in \mathcal{F}}\mathcal{O}'_{L^2}(V)\label{display:AnalyticFunctionalsCosheaf:IntroMathcalF} \\
    &\cong \colim_{\mathcal{F}} \mathcal{O}'_{L^2}(\cup \mathcal{F})& \text{(Corollary~\ref{corollary:L2AnalyticFunctionalsCosheafForFiniteCovers})} \label{equation:AnalyticFunctionalsCosheaf:ApplyCosheafForFiniteCover}\\
    &\cong \colim_{W \subset\subset Y} \mathcal{O}'_{L^2}(W)\label{display:AnalyticFunctionalsCosheaf:RemoveMathcalF} \\
    &\cong \mathcal{O}'(Y)\displaycomma &\text{(Proposition~\ref{proposition:AnalyticFunctionalOnSteinColimitOfL2s})}\label{display:AnalyticFunctionalsCosheaf:last}
  \end{align} 
where it remains to justify~(\ref{display:AnalyticFunctionalsCosheaf:IntroMathcalF}) and~(\ref{display:AnalyticFunctionalsCosheaf:RemoveMathcalF}).
Given~$U \in \mathcal{U}$ and~$V \subset\subset U$, we use~$\mathcal{F} = \{V\}$ to define the comparison map in~(\ref{display:AnalyticFunctionalsCosheaf:IntroMathcalF}).
For its inverse, we pick a~$U \in\mathcal{U}$ with~$V \subset\subset U$ given~$\mathcal{F}$ and~$V\in \mathcal{F}$; the resulting map from~$\mathcal{O}'_{L^2}(V)$ to the r.\,h.\,s.\ of~(\ref{display:AnalyticFunctionalsCosheaf:IntroMathcalF})
is independent of the choice of~$U$ because~$\mathcal{U}$ is closed under pairwise intersections. 
For the map in~(\ref{display:AnalyticFunctionalsCosheaf:RemoveMathcalF}), we use 
~$W = \cup \mathcal{F}$. 
For its inverse, let~$W \subset\subset \mathcal{F}$. Pick a finite subcover of~$\mathcal{U}$ of~$\overline{W}$ and shrink it to get a choice of~$\mathcal{F}$. 
This is possible because~$\overline{W}$ is compact. 
\end{proof}
We let~$X \otimes Y$ resp.\ $X \barotimes Y$ denote the tensor product of nuclear locally convex spaces~$X$ and~$Y$ resp.\ the completed tensor product of nuclear locally convex spaces.
We include a proof that dualizing is monoidal, i.\,e., dualizing takes the completed tensor products of nuclear locally convex spaces to the completed tensor product of complete bornological spaces, see Definition~\ref{definition:CompletedTensorProduct}. 
\begin{proposition}\label{proposition:DualOfTensorProductOfNuclearSpaces}
  Let~$X$ and~$Y$ be nuclear locally convex spaces.
  Then
  \MapMapsto{X' \barotimes Y'}{(X \barotimes Y)'}{\alpha \otimes \beta}{ \left[ x\otimes y \mapsto \alpha(x) \beta(y)\right]}
  is an isomorphism of nuclear convex bornological spaces. 
\end{proposition}
If~$X$ and~$Y$ are Hilbert spaces, then
  \begin{align*}
     X'\barotimes_{HS} Y' \cong (X \barotimes_{HS} Y)' 
  \end{align*}
  as Hilbert spaces where~$\barotimes_{HS}$ denotes the completed Hilbert-Schmidt tensor product.
  If~$X$ and~$Y$ are Banach spaces, then the bornology on the projective tensor product~$X \otimes_{\pi} Y$ agrees with the bornology on~$X \otimes Y$ from Definition~\ref{definition:BornologicalTensorProduct}, so the underlying bornological space of the completed projective tensor product of~$X$ and~$Y$ is the completed tensor product of the underlying complete bornological spaces of~$X$ and~$Y$.
  The completed tensor product of nuclear bornological spaces can be computed in terms of the completed Hilbert-Schmidt tensor product, similar to how the projective tensor product of nuclear locally convex topological vector spaces can be computed in terms of the Hilbert-Schmidt tensor product.

\begin{proof}  
  It suffices to prove that the comparison map~$c: X' \barotimes Y' \rightarrow (X \otimes Y)'$ is an isomorphism of bornological spaces, it then follows that the map~$X' \barotimes Y' \rightarrow (X \barotimes Y)'$ is an isomorphism because~$(X\barotimes Y)' \rightarrow (X \otimes Y)'$ is an isomorphism by the universal property of the completion~$X \barotimes Y$. 
  Letting~$p,q$ run over a family of semi-norms exhibiting~$X$ resp.~$Y$ as an inverse limit of Hilbert-Schmidt maps, 
  the map~$c$ factors as
  \begin{align*}
    X' \barotimes Y' &\cong \colim_p X'_p \barotimes \colim_q Y'_q\\
    &\cong \colim_{p,q} (X'_p \barotimes_{\pi} Y'_q) \\
    &\cong \colim_{p,q} (X'_p \barotimes_{HS} Y'_q) \\
    &\cong \colim_{p,q} (X_p \barotimes_{HS} Y_q)' \\
    &\cong \colim_{p,q} (X_p \otimes_{HS} Y_q)' \\
    &\cong \left(\lim_{p,q} (X_p \otimes_{HS} Y_q)\right)'
    \intertext{(since the underlying vector space of~$X_p \otimes_{HS} Y_q$ is independent of~$p$ and~$q$)}
    &\cong \left(X\otimes Y\right)'\displayperiod
  \end{align*} 
\end{proof}
Therefore, if~$U\subseteq \C^m$ and~$V\subseteq \C^n$, there is an isomorphism
\[
  \mathcal{O}'(U) \barotimes \mathcal{O}'(V) \cong (\mathcal{O}(U) \barotimes \mathcal{O}(V))' \cong \mathcal{O}(U\times V)' = \mathcal{O}'(U \times V)
\]
of complete bornological spaces.
This map sends~$\alpha \otimes \beta$ to a functional called~$\alpha \times \beta$, which on product functions~$f(z)g(w)$ of~$(z,w) \in U\times V$ is~$\alpha(f)\beta(g)$, and hence agrees with the map defined by Equation~\ref{equation:ExternalProductOfAnalyticFunctionals}.
A version of~$\mathcal{O}'(U) \barotimes \mathcal{O}'(V) \cong \mathcal{O}'(U\times V)$ in the context of locally convex spaces is stated in the Corollary to Theorem 51.6 of~\cite{Treves}.
The author has found~\cite[Chapter 4]{Helemskii} helpful for the proof of~$\mathcal{O}(U) \barotimes \mathcal{O}(V) \cong \mathcal{O}(U\times V)$, even though this book does not seem to cleanly state this isomorphism in the generality we use it. 

We now formulate the identity theorem for holomorphic functions with values in a complete bornological space and deduce it from the identity theorem for holomorphic functions with values in a Banach space. 
\begin{proposition}\label{proposition:BornologicalIdentityTheorem}
Let~$f$ be a holomorphic function on a domain~$D\subseteq \C$ with values in a complete bornological space.
If~$D$ contains an accumulation point of~$f^{-1}(0)$, then~$f$ is zero.
\end{proposition}
\begin{proof}
Let~$S$ be the set of points~$z \in D$ which have an open neighborhood~$U \subseteq D$ s.\,t.~$f|_U =0$. By definition~$S$ is open.
Locally,~$f$ admits a power series expansion in a Banachable subobject of~$X$ and thus
\[
S = \bigcap_{n \geq 0}\{ z \in D \mid f^{(n)}(z) = 0 \}\displayperiod
\]
Thus~$S$ is closed as the intersection of sets which are closed by the continuity of~$f$ and its derivatives w.\,r.\,t.\ the topology of b-closed sets.
It remains to show that~$S$ is not empty because the claim then follows since~$D$ is connected.
Let~$z_0 \in f^{-1}(0)$ be an accumulation point of~$f^{-1}(0)$.
After restricting~$f$ to an open neighborhood of~$z_0$, we may apply the identity theorem for holomorphic functions with values in a Banach space and conclude that~$z \in S$.
\end{proof}
\begin{corollary}\label{corollary:ToBornologicalIdentityTheorem}
  Let~$f$ be a holomorphic function on a domain~$D\subseteq \C$ with values in a complete bornological vector space~$X$.
  Let~$Y$ be a sub vector space of~$X$.
  If~$f^{-1}(Y)$ contains an accumulation point, then~$\im f \subseteq \overline{Y}$.
\end{corollary} 
\begin{proof}
  Let~$q:X\rightarrow X/\overline{Y}$ denote the quotient map.
  Its target~$X/\overline{Y}$ is the quotient of a complete bornological vector space by a b-closed subspace and hence complete. 
  Proposition~\ref{proposition:BornologicalIdentityTheorem} applies to~$q\after f$, so~$q\after f = 0$ and thus~$\im f \subseteq \overline{Y}$.
\end{proof} 

\Subsection{Multiplicativity}
The proof that~$E$ is multiplicative is a version of the proof that the symmetric algebra functor turns direct sums into tensor products.
\begin{proposition}\label{proposition:EIsMultiplicative}
  The prefactorization algebra~$E$ is multiplicative.
\end{proposition}
\begin{proof}
  Let~$U,V \subseteq \C$ be open and disjoint. 
  The multiplication map
  \[
    M^{U,V}_{U \sqcup V} : E(U) \barotimes E(V) \rightarrow E(U \sqcup V)
  \] 
  factors as
  \begin{align*}
    &E(U) \barotimes E(V) \\
    &= \left( \bigoplus_{m \geq 0} \left(\mathcal{O}'(U^m \setminus \Delta_m) \otimes \V^{\otimes m}\right)_{\Sigma_m} \right) \barotimes \left( \bigoplus_{n \geq 0} \left(\mathcal{O}'(V^n \setminus \Delta_n) \otimes \V^{\otimes n}\right)_{\Sigma_n} \right) \\
    &\cong \bigoplus_{m \geq 0} \bigoplus_{0 \leq i \leq m} \left(\mathcal{O}'(U^i \setminus \Delta_i) \otimes \V^{\otimes i}\right)_{\Sigma_i} \barotimes \left(\mathcal{O}'(V^{m-i} \setminus \Delta_{m-i}) \otimes \V^{\otimes (m-i)}\right)_{\Sigma_{m-i}} \\
&\cong \bigoplus_{m \geq 0} \left(\mathcal{O}'((U\sqcup V)^m \setminus \Delta) \otimes \V^{\otimes m}\right)_{\Sigma_m}\displaycomma
  \end{align*}
  where the first isomorphism holds because the completed tensor product commutes with direct sums separately in each variable. To deduce that the second map is an isomorphism, we fix~$m\geq 0$ and consider the map between the~$m$-th summands. 
  It factors as 
  \begin{align*}
    &\bigoplus_{0 \leq i \leq m} \left(\mathcal{O}'(U^i \setminus \Delta_i) \otimes \V^{\otimes i}\right)_{\Sigma_i} \barotimes \left(\mathcal{O}'(V^{m-i} \setminus \Delta_{m-i}) \otimes \V^{\otimes (m-i)}\right)_{\Sigma_{m-i}} \\
    &\cong\bigoplus_{0 \leq i \leq m} \left(\mathcal{O}'(U^i \setminus \Delta_i) \otimes \V^{\otimes i} \barotimes \mathcal{O}'(V^{m-i} \setminus \Delta_{m-i}) \otimes \V^{\otimes (m-i)}\right)_{\Sigma_i \times \Sigma_{m-i}} \\
    &\cong\bigoplus_{0 \leq i \leq m} \left(\mathcal{O}'(U^i \setminus \Delta_i)  \barotimes \mathcal{O}'(V^{m-i} \setminus \Delta_{m-i}) \otimes \V^{\otimes i} \otimes \V^{\otimes (m-i)}\right)_{\Sigma_i \times \Sigma_{m-i}} \\
    &\cong\bigoplus_{0 \leq i \leq m} \left(\mathcal{O}'((U^i \setminus \Delta_i)  \times (V^{m-i} \setminus \Delta_{m-i})) \otimes \V^{\otimes i} \otimes \V^{\otimes (m-i)}\right)_{\Sigma_i \times \Sigma_{m-i}} \\
    &\cong \left(\mathcal{O}'(U^m \setminus \Delta) \otimes \V^{\otimes m}\right)_{\Sigma_m}\displaycomma
  \end{align*}
  where the last isomorphism requires further justification.
Let~$\mathbf{m} = \{1,\ldots, m\}$.
For finite sets~$I$, in particular~$I\subseteq \mathbf{m}$, we have the finite product~$\C^I$ and its subset
\[
  \Delta_I := \{ z \in \C^I \mid z_i = z_j \text{ for some } i,j \in I \text{ with } i \neq j \}\displayperiod
\]
There are~$\Sigma_m$-equivariant complex-analytic isomorphisms
\begin{align*}
(U \sqcup V)^m \setminus \Delta_m &= (U \sqcup V)^\mathbf{m} \setminus \Delta_\mathbf{m}\\
&\cong \bigsqcup_{I \subseteq \mathbf{m}} (U^I \times V^{\mathbf{m} \setminus I }) \setminus \Delta_I\\
& \cong \bigsqcup_{I \subseteq \mathbf{m}} (U^I \setminus \Delta_I) \times (V^{\mathbf{m} \setminus I } \setminus \Delta_{\mathbf{m} \setminus I }) & \text{(}U \cap V = \emptyset\text{)}
\end{align*}
given by the inclusions and associativity isomorphism.
Combining the induced isomorphism on~$\mathcal{O}'$ with the isomorphism~$\mathcal{O}'(X) \barotimes \mathcal{O}'(Y) \cong \mathcal{O}'(X \times Y)$ for~$X,Y$ open subsets of finite-dimensional complex vector spaces, we get
\begin{align*}
  \bigoplus_{I \subseteq \mathbf{m}} \mathcal{O}'(U^I \setminus \Delta_I) \barotimes \mathcal{O}'(V^{\mathbf{m} \setminus I } \setminus \Delta_{\mathbf{m} \setminus I })&\cong \bigoplus_{I \subseteq \mathbf{m}} \mathcal{O}'((U^I \setminus \Delta_I) \times (V^{\mathbf{m} \setminus I } \setminus \Delta_{\mathbf{m} \setminus I })) \\
  &\cong \mathcal{O}'((U \sqcup V)^m \setminus \Delta_m)\displaycomma
\end{align*}
again~$\Sigma_m$-equivariantly.
Tensoring with the~$\Sigma_m$-representation~$\V^{\mathbf{m}}$ gives a~$\Sigma_m$-isomorphism
\begin{align*}
  \bigoplus_{I \subseteq \mathbf{m}} \left(\mathcal{O}'(U^I \setminus \Delta_I) \otimes \V^I\right) \barotimes \left(\mathcal{O}'(V^{\mathbf{m} \setminus I } \setminus \Delta_{\mathbf{m} \setminus I }) \otimes \V^{\mathbf{m} \setminus \mathbf{I}}\right)\\
  \cong \mathcal{O}'((U \sqcup V)^m \setminus \Delta_m) \otimes \V^m\displayperiod  
\end{align*}
This yields the second isomorphism in  
\begin{align*}
&  \bigoplus_{0 \leq i \leq m} \left(\mathcal{O}'(U^i \setminus \Delta_i) \otimes \V^{\otimes i}\right)_{\Sigma_i} \barotimes \left(\mathcal{O}'(V^{m-i} \setminus \Delta_{m-i}) \otimes \V^{\otimes (m-i)}\right)_{\Sigma_{m-i}}\\
  &\cong\bigoplus_{0 \leq i \leq m} \left(\left(\mathcal{O}'(U^i \setminus \Delta_i) \otimes \V^i\right) \barotimes \left(\mathcal{O}'(V^{m - i} \setminus \Delta_{\mathbf{m} -i }) \otimes \V^{m-i}\right)\right)_{\Sigma_i \times \Sigma_{m-i}}\\
  &\cong \left(\mathcal{O}'((U \sqcup V)^m \setminus \Delta_m) \otimes \V^m \right)_{\Sigma_m} \displaycomma
\end{align*}
and the composite is the map on the~$m$-th summand. 
\end{proof}
\begin{proposition}\label{proposition:FIsMultiplicative}
  The prefactorization algebra~$F$ is multiplicative.
\end{proposition}
\begin{proof}
  Let~$U, V \subseteq \C$ be open and disjoint.
  Recall from the proof of Proposition~\ref{proposition:FIsAPrefactorizationAlgebra} that the map~$E(U) \barotimes E(V) \rightarrow F(U) \barotimes F(V)$ is the cokernel of the multiplication map
  \[
    T = S^{disc}(U) \barotimes E(V) \oplus E(U) \barotimes S^{disc}(V) \rightarrow E(U) \barotimes E(V)\displayperiod
  \]
  We already know from Proposition~\ref{proposition:EIsMultiplicative} that~$E(U) \barotimes E(V) \cong E(U\sqcup V)$ via the same map that induces the multiplication map~$F(U) \barotimes F(V) \rightarrow F(U \sqcup V)$ given by Proposition~\ref{proposition:FIsAPrefactorizationAlgebra}.
  Therefore, it remains to see that the cokernel of~$E(U\sqcup V)$ by~$S^{disc}(U\sqcup V)$ agrees with the cokernel of
  \[
    T \rightarrow E(U) \barotimes E(V) \rightarrow E(U \sqcup V)\displayperiod
  \]  
  For this, we use that there is an isomorphism
  \begin{align}
    S^{disc}(U\sqcup V ) \cong S^{disc}(U) \barotimes \left(\bigoplus_{W' \subseteq V} E(W')\right) \oplus \left(\bigoplus_{W' \subseteq U} E(W')\right) \barotimes S^{disc}(V)\label{equation:SdiscUVAlternative}
  \end{align}
  compatible with the multiplication maps to~$E(U \sqcup V)$. Using~(\ref{equation:SdiscUVAlternative}), we see that~$S^{disc}(U \sqcup V)$ has maps to and from~$T$ compatible with the maps to~$E(U \sqcup V)$. 
  Hence the images of~$S^{disc}(U \sqcup V)$ and~$T$ in~$E(U\sqcup V)$ agree.
  We obtain the isomorphism in~(\ref{equation:SdiscUVAlternative}) from the fact that any disc~$d \subseteq U \sqcup V$ is contained in exactly one of~$U$ and~$V$ and from the multiplicativity of~$E$,
  {\allowdisplaybreaks\begin{align*}
    S^{disc}(U\sqcup V ) &\cong \bigoplus_{\stackrel{d \subseteq U}{W \subseteq (U\sqcup V) \setminus d}} E(W) \barotimes R(d)  \oplus \bigoplus_{\stackrel{d \subseteq V}{W \subseteq (U\sqcup V) \setminus d}} E(W) \barotimes R(d) \\
    &\cong \bigoplus_{\stackrel{d \subseteq U}{W \subseteq (U\sqcup V) \setminus d}} E(W\cap V) \barotimes E(W\cap U) \barotimes R(d) \\
    &\quad\quad\oplus \bigoplus_{\stackrel{d \subseteq V}{W \subseteq (U\sqcup V) \setminus d}} E(W\cap U) \barotimes E(W \cap V) \barotimes R(d) \\
    &\cong \bigoplus_{\stackrel{d \subseteq U}{W \subseteq U\setminus d, W' \subseteq V}} E(W') \barotimes E(W) \barotimes R(d)\\
    &\quad\quad  \oplus \bigoplus_{\stackrel{d \subseteq V}{W \subseteq V \setminus d, W' \subseteq U}} E(W') \barotimes E(W) \barotimes R(d) \\
    &\cong \left(\bigoplus_{\stackrel{d \subseteq U}{W \subseteq U\setminus d}}  E(W) \barotimes R(d) \right) \barotimes  \left(\bigoplus_{W' \subseteq V} E(W')\right)\\
    &\quad\quad\oplus\left(\bigoplus_{W' \subseteq U}  E(W')\right) \otimes \left( \bigoplus_{\stackrel{d \subseteq V}{W \subseteq V \setminus d }} E(W) \barotimes R(d)\right) \\
    &\cong S^{disc}(U) \barotimes \left(\bigoplus_{W' \subseteq V} E(W')\right) \oplus \left(\bigoplus_{W' \subseteq U} E(W')\right) \barotimes S^{disc}(V)\displayperiod
\end{align*}}
\end{proof}

\Subsection{Weiss Codescent}
We prove that~$E$ and~$F$ are Weiss cosheaves in this section and state a more detailed version of the theorem from the introduction.  
\begin{proposition}\label{proposition:EIsAWeissCosheaf}
The precosheaf~$E$ is a Weiss cosheaf of complete bornological vector spaces.
\end{proposition}
\begin{proof}
Let~$X$ be an open subset of~$\C$ and~$\mathcal{U}$ be a Weiss cover of~$X$.
There is a commutative diagram
\begin{ctikzcd}
\colim\limits_{U \in \mathcal{U}} E(U) \rar{} &  E(X) \\
\bigoplus\limits_{n\geq 0} \left(\colim\limits_{U \in \mathcal{U}} \left(\mathcal{O}'(U^n \setminus \Delta)\right) \otimes V^{\otimes n}\right)_{\Sigma_n} \rar{} \ar[u, "\cong"]&  \bigoplus\limits_{n\geq 0} \left(\mathcal{O}'(X^n \setminus \Delta) \otimes V^{\otimes n}\right)_{\Sigma_n} \ar[u,equal]&
\end{ctikzcd}
in~$\BVS$. 
The left hand map is an isomorphism because colimits commute with each other and with the completed tensor product.
The bottom map is an isomorphism because~$\mathcal{O}'$ is a cosheaf for Stein open covers of Stein subsets of~$\C^n$ for each~$n\geq 0$ by Proposition~\ref{proposition:AnalyticFunctionalsFormACosheafForSteinCoversOfSteinSubsetsOfCn} and~$U^n \setminus\Delta$ is Stein for all open subsets~$U \subseteq \C$.
Since~$\mathcal{U}$ is a Weiss cover,~$\{U^n\}_{U\in \mathcal{U}}$ is an open cover of~$X^n$ for each~$n\geq 0$.
The cosheaf property applied to the open cover~$\{U^n\setminus \Delta\}_{U\in \mathcal{U}}$ of~$X^n \setminus \Delta$ implies that the bottom map is an isomorphism.
\end{proof}
The following lemma says that the set of relations on a disc is the closure of relations on every smaller concentric disc. 
\begin{lemma}\label{lemma:RelationsOnSmallerConcentricDiscAreDenseInRelationsOnDisc}
  Let~$d$ be a disc of radius~$r$ with center~$z_0 \in \C$ and let~$\delta > 0$ with~$\delta \leq r$. 
  Then
  \[
    R(d) = \overline{i\left(R(B_\delta(z_0))\right)} \subseteq E(d)\displaycomma
  \]
  where~$i = E(B_\delta(z_0) \hookrightarrow d)$ and the b-closure is taken in~$E(d)$.
\end{lemma}
Roughly speaking, the proof of the lemma involves viewing an element of~$R(d)$ as the endpoint of a path~$[\varepsilon,1] \rightarrow R(d)$ which extends to a holomorphic function taking values in~$i(R(B_\delta(z_0)))$ near its beginning. 
The path is given by scaling down a given relation, and this preserves the property of being a relation, i.\,e., evaluating to zero, because the evaluation map is equivariant. 
It then follows that the holomorphic function has image in the b-closure of the image of~$R(B_\delta(z_0))$ by Corollary~\ref{corollary:ToBornologicalIdentityTheorem} to the identity theorem for holomorphic functions. 
\begin{proof}[Proof of Lemma~\ref{lemma:RelationsOnSmallerConcentricDiscAreDenseInRelationsOnDisc}]
  It suffices to consider the case of~$z_0 = 0$ by the translation invariance of~$E$.
  The inclusion~``$\supseteq$'' holds because~$R(d)$ is b-closed in~$E(d)$ by Remark~\ref{remark:ClosednessOfRelations}.
  To prove~``$\subseteq$'', let~$x\in R(d)$.
  There exists a strictly smaller concentric disc~$d' \subset d$ and an~$x' \in E(d')$ which is mapped to~$x$ by the map induced by the inclusion of~$d'$ into~$d$ because~$E$ is a Weiss cosheaf and the set of concentric discs strictly smaller than~$d$ forms a Weiss cover of~$d$. 
  We have~$x' \in R(d')$ because the evaluation map is compatible with inclusions. 
  Let~$r >r' > 0$ denote the radius of~$d$ resp.~$d'$.
  We can restrict attention to the case that~$d'$ strictly contains~$B_\delta(0)$, that is, that~$r' > \delta$, because otherwise~$d' \subseteq B_\delta(0)$ and thus~$x \in i(R(B_\delta(0)))$. 
  Let~$R = \frac{r}{r'}$ and~$A = B_R(0) - 0 \subseteq \C^\times$.
  From now on, we only consider rotations and dilations among the affine-linear isomorphisms of~$\C$, and we identify~$A$ with a subset of the interior of~$\mathcal{D}_{d',d}$ via the group homomorphism~$\C^\times \hookrightarrow \C^\times \ltimes \C$.
  The map 
  \labelledMapMapsto{f}{A}{E(d)}{q}{\rho_{d',d}(q)(x')}
  is holomorphic because~$E$ is holomorphic. 
  Note that~$1\in A$ since~$R>1$ because~$d'$ is strictly smaller than~$d$.
  Also note that~$f(1)$ is the image of~$x'$ under the inclusion of~$d'$ into~$d$ and hence~$f(1) = x$.
  Let~$N$ be the set of~$q \in A$ with~$|q| < \frac{\delta}{r'}$.
  If~$q \in N$, then~$q.d' \subseteq B_\delta(0)$, so that~$q \in \mathcal{D}_{d',B_\delta(0)}$ and
  \[
    \rho_{d',d}(q) = \rho_{B_\delta(0), d}(1) \after \rho_{d', B_\delta(0)}(q) = i \after \rho_{d', B_\delta(0)}(q)\displayperiod
  \] 
  The evaluation map is equivariant by Proposition~\ref{proposition:EvaluationIsEquivariant}, so
  \[
    f(q) = \rho_{d',d}(q)(x') \in i(R(B_\delta(0)))
  \]
  for~$q\in N$. 
  Since~$N\subseteq A$ is non-empty and open, it has an accumulation point inside~$A$.
  Corollary~\ref{corollary:ToBornologicalIdentityTheorem} applies to the holomorphic function~$f$ defined on the domain~$A$, the complete bornological vector space~$E(d)$, and its subspace~$i(R(B_\delta(0)))$ to which~$N$ maps. 
  Therefore, the image of~$f$ is contained in~$\overline{i(R(B_\delta(0)))}$.
  In particular,~$x=f(1)$ is an element of this set.
\end{proof}
\begin{definition}
  Let~$X$ be an open subset of~$\C$ and let~$\mathcal{U}$ be a Weiss cover of~$X$.
  The \emph{relation on discs subordinate to~$\mathcal{U}$} are defined as
  \[
    R^{disc}_\mathcal{U}(X) = \im \left( \bigoplus_{d,W} R(d) \otimes E(W)  \longrightarrow E(X) \right)
  \]
  where the direct sum is taken over pairs of a disc~$d\subseteq X$ and~$W\subseteq X$ open and disjoint from~$d$ such that there exists a~$U \in \mathcal{U}$ s.\,t.~$d$ and~$W$ are both subsets of~$U$.  
\end{definition}
The last condition in the definition of~$R^{disc}_\mathcal{U}(X)$, namely that~$d$ and $W$ are contained in a single~$U \in \mathcal{U}$, is the only difference from the definition of~$R^{disc}(X)$. 
\begin{lemma}\label{lemma:RelationsSubordinateToCover}
  If~$X$ is an open subset of~$\C$ and~$\mathcal{U}$ a Weiss cover of~$X$, then
  \[
    R^{disc}(X) \subseteq \overline{R^{disc}_\mathcal{U}(X)} \displaycomma
  \]
  where the b-closure is taken in~$E(X)$. 
\end{lemma}
\begin{proof}
  By definition, the vector space~$R^{disc}(X)$ is generated by elements of the form~$M^{d,W}_X(f\otimes g)$ for~$f\in R(d)$ and~$g\in E(W)$ for some disc~$d \subseteq X$ and~$W\subseteq X$ open and disjoint from~$d$.   
  Let~$z_0$ be the center of~$d$.
  The set
\[
\mathcal{W} = \{ V \subseteq W \text{ open} \mid \exists U \in \mathcal{U} : U \cap W = V, z_0 \in U \}
\]
is a Weiss cover of~$W$ as we now check.
Let~$S\subseteq W$ be finite.
Since~$S \cup \{z_0\}$ is a finite subset of~$X$, there is a~$U\in \mathcal{U}$ s.\,t.~$S\cup \{z_0\} \subseteq U$.
Thus
\[
  V:= U \cap W \supseteq S  \displayperiod
\]
By the definition of the subspace topology,~$V$ is an open subset of~$W$. This concludes the proof that~$\mathcal{W}$ is a Weiss cover of~$W$.

Let~$f\otimes g \in R(d) \otimes E(W)$.
Since~$\mathcal{W}$ is a Weiss cover of~$W$ and~$E$ is a Weiss cosheaf, there are~$W_1,\ldots W_k \in \mathcal{W}$ and~$g_i \in \im E(W_i\hookrightarrow W)$ for~$i=1,\ldots, k$ s.\,t.
\[
g = \sum^k_{i=1} g_i\displaycomma
\]
and thus
\[
f \otimes g = \sum^k_{i=1} f\otimes g_i \displayperiod
\]
It suffices to consider each summand individually because~$\overline{R^{disc}_\mathcal{U}(X)}$ is closed under addition because it is a sub vector space as the b-closure of a sub vector space, see~\cite[2:12 Proposition 1]{HogbeNlend}. 
We consider each of these summands individually and may therefore restrict attention to the case that there is a~$U\in \mathcal{U}$ s.\,t.~$g \in \im E(U \cap W \hookrightarrow W)$ and~$z_0 \in U$. 
Our goal is to prove that~$M^{d,W}_X(f \otimes g) \in \overline{R^{disc}_\mathcal{U}(X)}$. 
There is a~$\delta>0$ s.\,t.~$B_\delta(z_0) \subseteq d \cap U$ because both~$U$ and~$d$ are open and contain~$z_0$.
Lemma~\ref{lemma:RelationsOnSmallerConcentricDiscAreDenseInRelationsOnDisc} says that~$R(d)= \overline{i(R(B_\delta(z_0)))}$ where the b-closure is taken in~$E(d)$ and~$i$ is the map from~$E(B_\delta(z_0))$ to~$E(d)$ induced by the inclusion of~$B_\delta(z_0)$ into~$d$.
Thus~$f\otimes g \in \overline{i(R(B_\delta(z_0)))} \otimes E(W)$.
The claim follows because the map~$M^{d,W}_X(\blank \otimes g) : E(d) \rightarrow E(X)$ is bounded and hence continuous, so
\begin{align*}
  M^{d,W}_X(f \otimes g) \in &M^{d,W}_X(\overline{i(R(B_\delta(z_0)))} \otimes g)\\
  &\subseteq \overline{M^{d,W}_X(i(R(B_\delta(z_0))) \otimes g)} \subseteq \overline{R^{disc}_\mathcal{U}(X)}\displayperiod
\end{align*}
\end{proof}
\begin{theorem}\label{theorem:FIsAWeissCosheaf}
  The precosheaf~$F$ is a Weiss cosheaf of complete bornological vector spaces. 
\end{theorem}
\begin{proof}
  Let~$X$ be an open subset of~$\C$ and let~$\mathcal{U}$ be a Weiss cover of~$X$ closed under taking pairwise intersections.
  Our goal is to show that the right vertical map in the following commutative diagram, whose maps are explained below, is an isomorphism.
  \begin{ctikzcd}
    \colim\limits_{U\in \mathcal{U}} R^{disc}(U) \rar{i_\mathcal{U}} \dar{\gamma_{R}} & \colim\limits_{U \in \mathcal{U}}E(U)  \rar{q_\mathcal{U}} \dar{\gamma_{E}} & \colim\limits_{U \in \mathcal{U}}F(U) \dar{\gamma_{F}}\\
     R^{disc}(X) \rar{i_X} & E(X)  \rar{q_X} & F(X) 
  \end{ctikzcd}
  The colimits are taken over~$\mathcal{U}$. 
  The vertical maps are induced by inclusions between subsets of~$\C$.
  Let
  \[
    q_X: E(X) \rightarrow E(X)/\overline{R^{disc}(X)} = F(X)
  \]
  be the quotient map and let
  \[
    q_\mathcal{U} : \colim_{U \in \mathcal{U}}E(U) \longrightarrow \colim_{U \in \mathcal{U}}F(U)
  \]
  be the map induced by the quotient maps~$q_U$ for~$U \in \mathcal{U}$ instead of~$X$.
  In the lower row, the map~$i_X$ is the inclusion, and the second map~$q_X$ is its cokernel in the category of complete bornological  spaces. 
  The upper row is the colimit of the analogue of the lower row for~$U\in \mathcal{U}$ instead of~$X$.
  In the upper row, it also holds that the second map is the cokernel of the first, because colimits commute with each other.

  The map~$\gamma_E$ is an isomorphism because~$E$ is a cosheaf by Proposition~\ref{proposition:EIsAWeissCosheaf}.
  Hence~$\gamma_F$ is an isomorphism if~$\overline{\im i_X} = \overline{\gamma_E (\im (i_\mathcal{U}))}$ where the b-closure is taken in~$E(X)$, because cokernels in~$\CBVS$ are computed by modding out the b-closure of the image. 
  Note that~$\im i_X = R^{disc}(X)$ by definition. On the other side,
  \[
    \gamma_E (\im (i_\mathcal{U})) = \im \left( \bigoplus_{d,W} R(d) \otimes E(W)  \longrightarrow E(X) \right) = R^{disc}_\mathcal{U}(X)\displaycomma
  \]
  where~$d\subseteq X$ is a disc and~$W\subseteq X$ is open and disjoint from~$d$ such that there exists a~$U \in \mathcal{U}$ s.\,t.~$d$ and~$W$ are both subsets of~$U$, and
  the last equality is the definition of~$R^{disc}_\mathcal{U}(X)$.
  Therefore, it remains to show~$\overline{R^{disc}(X)} = \overline{R^{disc}_\mathcal{U}(X)}$.
  The inclusion~``$\supseteq$'' follows from~$R^{disc}(X) \supseteq R^{disc}_\mathcal{U}(X)$ which is consequence of the definitions. 
  The inclusion~``$\subseteq$'' follows from Lemma~\ref{lemma:RelationsSubordinateToCover}.
\end{proof}
\begin{theorem}\label{theorem:MoreDetailedMainTheorem}
  If~$\V$ is a geometric vertex algebra, then~$\mathbf{F}\V$ is a holomorphic factorization algebra with meromorphic OPE and discrete weight spaces whose associated geometric vertex algebra~$\mathbf{V}\mathbf{F}\V$ is isomorphic to~$\V$. 
\end{theorem}
\begin{proof}
  Proposition~\ref{proposition:FIsAPrefactorizationAlgebra} says that~$F =\mathbf{F}\V$ is a prefactorization algebra.
  Proposition~\ref{proposition:EandFareHolomorphic} says that~$F$ is holomorphic. 
  Proposition~\ref{proposition:VFisIdAsGradedBVS} says that~$\mathbf{V}\mathbf{F}\V \cong \V$ as graded bornological vector spaces. In particular,~$F$ has discrete weight spaces.   
  Proposition~\ref{proposition:VFisIdAsGeometricVertexAlgebras} says that this isomorphism respects the multiplication maps, and~$\mathbf{F}\V$ having meromorphic OPE by definition means that~$\mathbf{V}\mathbf{F}\V$ satisfies the meromorphicity axiom of a geometric vertex algebra.  
  Proposition~\ref{proposition:FIsMultiplicative} says that~$F$ is multiplicative, and 
  Theorem~\ref{theorem:FIsAWeissCosheaf} says that~$F$ is a Weiss cosheaf, so~$F$ is a factorization algebra. 
\end{proof}